\documentclass{amsart}
\usepackage[utf8]{inputenc}
\usepackage[english]{babel}

\usepackage{csquotes}
\usepackage[backend=biber]{biblatex}
\usepackage{xcolor}
\usepackage{amssymb}
\usepackage{graphicx}
\usepackage{enumerate}

\title[An embedding theorem for subshifts over groups with comparison]{An embedding theorem for subshifts over amenable groups with the comparison property}
\author{Robert Bland}

\newtheorem{theorem}{Theorem}[section]
\newtheorem{lemma}[theorem]{Lemma}

\newtheorem*{main-result}{Theorem~\ref{thm:main-result}}

\theoremstyle{definition}
\newtheorem{definition}[theorem]{Definition}

\newcommand{\A}{\mathcal A}
\newcommand{\F}{\mathcal F}
\renewcommand{\P}{\mathcal P}
\newcommand{\T}{\mathcal T}
\renewcommand{\S}{\mathcal S}

\DeclareMathOperator{\opint}{int}
\renewcommand{\int}{\opint}

\DeclareMathOperator{\ret}{ret}
\DeclareMathOperator{\ex}{ex}

\allowdisplaybreaks[4]

\begin{document}

\maketitle

\begin{abstract}
    We obtain the following embedding theorem for symbolic dynamical systems. Let $G$ be a countable amenable group with the comparison property. Let $X$ be a strongly aperiodic subshift over $G$. Let $Y$ be a strongly irreducible shift of finite type over $G$ which has no global period, meaning that the shift action is faithful on $Y$. 
    If the topological entropy of $X$ is strictly less than that of $Y$, and $Y$ contains at least one factor of $X$, then $X$ embeds into $Y$. This result partially extends the classical result of Krieger when $G = \mathbb{Z}$ and the results of Lightwood when $G = \mathbb{Z}^d$ for $d \geq 2$. The proof relies on recent developments in the theory of tilings and quasi-tilings of amenable groups.
\end{abstract}

\section{Introduction}


Our central question is as follows. Given subshifts $X$ and $Y$ over a countable amenable group $G$, under what conditions does $X$ \textit{embed} into $Y$? That is, under what conditions is $X$ isomorphic to a subsystem of $Y$? One necessary condition is that $h(X) \leq h(Y)$, where $h(X)$ is the \textit{topological entropy} of the subshift $X$ (Definition~\ref{def:entropy}). When $G = \mathbb{Z}$, the classical embedding theorem of Krieger \cite[Theorem~3]{krieger} provides a complete answer in the case that $Y$ is a mixing shift of finite type (SFT) and $h(X) < h(Y)$. 
Namely, a certain necessary condition about periodic points turns out to be sufficient for an embedding $\psi : X \to Y$ to exist. 
In particular, the condition is automatically satisfied if $X$ is \textit{strongly aperiodic} (Definition~\ref{def:aperiodic}), meaning that no point of $X$ exhibits a non-identity element of $G$ as a period (in other words, the shift action is \textit{free} on $X$).
Krieger's embedding theorem has become a cornerstone of the structure theory of SFTs over $\mathbb{Z}$.

Much less is known about the embedding problem for groups other than $G = \mathbb{Z}$. In the case where $G = \mathbb{Z}^d$ for $d\geq 2$, one partial result is given by Lightwood: suppose $X$ is a strongly aperiodic subshift and suppose $Y$ is an SFT which satisfies a mixing condition (``square mixing", known elsewhere as the ``uniform filling" property) and contains a point with a finite orbit. If $h(X) < h(Y)$ and $Y$ contains at least one factor of $X$, then $X$ embeds into $Y$ \cite[Theorem~2.5]{lightwood}. 
In the case where $G = \mathbb{Z}^2$, Lightwood proved that a square mixing SFT automatically contains a point with a finite orbit \cite[Lemma~9.2]{lightwood}. Later, Lightwood proved that if $Y$ is an SFT over $\mathbb{Z}^2$  satisfying a slightly stronger mixing condition (``square-filling mixing", which implies square mixing), then $Y$ automatically contains at least one factor of $X$ \cite[Theorem~2.8]{lightwood-II}. These results together provide a partial extension of Krieger's embedding theorem to $G = \mathbb{Z}^2$. 

In this paper, we obtain an embedding theorem for subshifts over countable amenable groups with the \textit{comparison property}. We do not define the comparison property here, but we appeal to it in the form of Theorem~\ref{thm:comparison-injection}, a characterization of the comparison property due to Downarowicz and Zhang \cite{downarowicz-zhang-symb-ext, downarowicz-zhang-comparison}. We also note that the class of amenable groups with the comparison property includes every countable group containing no finitely generated subgroup of exponential growth (proof given in both \cite[Theorem~6.33]{downarowicz-zhang-symb-ext} and \cite[Theorem~5.11]{downarowicz-zhang-comparison}). In particular, this includes every countable abelian group. It is unknown whether there exists a countable amenable group without the comparison property.

Our main result is as follows.

\begin{main-result}
    Let $G$ be a countable amenable group with the comparison property. Let $X$ be a nonempty strongly aperiodic subshift over $G$. Let $Y$ be a strongly irreducible SFT over $G$ with no global period. 
    If $h(X) < h(Y)$ and $Y$ contains at least one factor of $X$, then $X$ embeds into $Y$.
\end{main-result}

If one selects $G = \mathbb{Z}^d$ for $d\geq 2$ in the above statement, then one does not immediately recover the theorem of Lightwood. We assume here a slightly stronger mixing condition on $Y$ (strong irreducibility in place of square mixing). But, we do \textit{not} assume that $Y$ contains a point with a finite orbit; instead, we assume that $Y$ has no global period, 
a condition which is automatically true for strongly irreducible subshifts over $\mathbb{Z}^d$. 


The condition that $Y$ has no global period means that the shift action is faithful on $Y$. This condition is examined in detail in \textsection\ref{sect:finding-Y_0}. This condition is necessary for the theorem; in particular, if a nonempty strongly aperiodic subshift $X$ embeds into $Y$, then $Y$ exhibits an aperiodic point and therefore has no global period.




The condition that $X$ is strongly aperiodic allows us to derive systems of arbitrarily nice \textit{quasi-tilings} (Definition~\ref{def:quasi-tiling}) of the group $G$ as factors of $X$, by appealing to a theorem of Downarowicz and Huczek \cite[Lemma~3.4]{downarowicz-huczek}. The strong aperiodicity is necessary for this; indeed, if $X$ factors onto systems of quasi-tilings with arbitrarily large, disjoint tiles, then no point of $X$ can exhibit a non-trivial period.
The comparison property allows us to go one step further and derive from $X$ systems of arbitrarily nice \textit{exact tilings} of the group $G$. We accomplish this by adapting a construction of Downarowicz and Zhang (presented in both \cite[Theorem~6.3]{downarowicz-zhang-comparison} and \cite[Theorem~7.5]{downarowicz-zhang-symb-ext}).


The condition that $h(X) < h(Y)$ allows us to deduce that if a finite subset (namely, the shape of a tile in a given quasi-tiling) is large enough, then there are more patterns of that shape appearing in points of $Y$ than in $X$.
This implies that there is an injective map from tile patterns in $X$ to tile patterns in $Y$. The condition that $Y$ is a strongly irreducible SFT allows us to mix those tile patterns together into a single point of $Y$. The condition that $Y$ contains a factor of $X$ provides us with a homomorphism $\phi : X \to Y$ (not necessarily injective) which we use to code the boundaries of the tiles.

The condition that $Y$ has no global period allows \textit{marker} patterns to be constructed in $Y$ (Theorem~\ref{thm:the-marker-patterns}), which are used to encode the locations of the centers of the tiles of a given quasi-tiling within a controllably-sparse subset of the symbols of a point of $Y$. The marker patterns allow one to uniquely decode, from a given image point $y = \psi(x)$, which quasi-tiling $t = \T(x)$ was used to construct $y$ to begin with. Then, the tile pattern injections from earlier allow one to uniquely reconstruct the preimage $x$.

We note that the class of strongly aperiodic subshifts is well populated. Indeed, Rosenthal \cite[Theorem~2']{rosenthal} has shown that every free ergodic measure preserving action of a discrete amenable group $G$ is measurably isomorphic to a strictly ergodic subshift over $G$. From this we see the existence of a large supply of strongly aperiodic subshifts. This was also observed by Lightwood \cite{lightwood} for $G = \mathbb{Z}^2$.

We also mention here the similarity of hypotheses between the present work and contemporary work by Huczek and Kopacz \cite{huczek-kopacz} which considers the ``factor problem" (under what hypotheses on $X$ and $Y$ does there exist a surjective homomorphism from $X$ to $Y$?) for subshifts over discrete amenable groups, which is complementary to the embedding problem. In particular, \cite[Theorem~2.12]{huczek-kopacz} utilizes the comparison property of $G$ in a form similar to Theorem~\ref{thm:quasi-to-exact-tiling-factor}. Additionally, \cite[Theorem~5.1]{huczek-kopacz} assumes that the domain $X$ is strongly irreducible and that $X$ exhibits, for each finite subset $F\subset G$, a pattern which exhibits no element of $F$ as a period. This latter condition is (under assumption of strong irreducibility) equivalent to our ``separating elements" condition (Definition~\ref{def:separate-elements}), which is the form in which we appeal to the faithfulness of the shift action of $G$ on $Y$.

The paper is organized as follows. In \textsection 2, we review some preliminary material about countable amenable groups, symbolic dynamics, and quasi-tilings over countable amenable groups with and without the comparison property. In \textsection 3 we present our main results. We first construct a subsystem $Y_0 \subset Y$ which we pass to in the construction of our embedding (Theorem~\ref{thm:finding-the-good-target-Y0}). We then construct marker patterns for $Y$ (Theorem~\ref{thm:the-marker-patterns}). Finally, we present the construction of our embedding of $X$ into $Y$ (Theorem~\ref{thm:main-result}). In \textsection 4, we discuss various ways in which Theorem~\ref{thm:main-result} could potentially be strengthened and associated obstacles.


\section{Prelimaries}

\subsection{Amenable groups}

In this section, we briefly review the theory of countable amenable groups and state a few lemmas which shall be needed later. For a more thorough introduction to the theory of dynamics on amenable groups, see \cite{kerr-li}.

\begin{definition}[Invariance, amenability, and F\o lner sequences]
\label{def:folner-seq}
    Let $G$ be a countable group, let $K \subset G$ be a finite subset and let $\varepsilon > 0$. A finite subset $F \subset G$ is said to be \textit{$(K,\varepsilon)$-invariant} if \[
        |KF\triangle F| < \varepsilon |F|.
    \] The group $G$ is \textit{amenable} if, for any finite subset $K \subset G$ and $\varepsilon > 0$, there exists a finite subset $F \subset G$ which is $(K,\varepsilon)$-invariant. Equivalently, $G$ is amenable if there exists a sequence $(F_n)_n$ of finite subsets of $G$ such that for each fixed finite subset $K \subset G$, it holds that \[
        \lim_{n\to\infty} \frac{|KF_n \triangle F_n|}{|F_n|} = 0.
    \] In this case, we say that $(F_n)_n$ is a \textit{F\o lner sequence}.
\end{definition}

Throughout this paper, $G$ denotes a fixed countable amenable group with identity element $e$. It is classically known \cite[Theorem~5.2]{namioka} that $G$ exhbits a F\o lner sequence $(F_n)_n$ that is ascending ($F_n \subset F_{n+1}$ for each $n$), that satisfies $\bigcup_n F_n = G$, and such that each $F_n$ is \textit{symmetric} ($F_n^{-1} = F_n$). Throughout this paper, $(F_n)_n$ denotes a fixed F\o lner sequence with all of the above properties. The symmetry property implies that $(F_n)_n$ is both left F\o lner and right F\o lner, because for each finite subset $K \subset G$ we have that \[
    |KF_n\triangle F_n| = |(KF_n\triangle F_n)^{-1}| = |F_n^{-1}K^{-1}\triangle F_n^{-1}| = |F_nK^{-1}\triangle F_n|.
\]

In this setting, when a finite subset $F \subset G$ is described as ``large", it is implied that $F$ is $(K,\varepsilon)$-invariant for some finite subset $K \subset G$ and $\varepsilon > 0$, which may be clear from context or chosen arbitrarily beforehand.
This sort of terminology is common but vague; a precise formulation is as follows. Let $\phi(F)$ be a property of finite subsets $F\subset G$. We shall say $\phi(F)$ holds for all \textit{sufficiently large} $F$ if there is a $(K,\varepsilon)$ such that if $F$ is $(K,\varepsilon)$-invariant then $\phi(F)$ is true. We shall say $\phi(F)$ holds for \textit{arbitrarily large} subsets $F$ if for every $(K,\varepsilon)$ there is a set $F$ for which $F$ is $(K,\varepsilon)$-invariant and $\phi(F)$ is true. 

We shall need the following elementary lemma; we omit the proof for brevity.

\begin{lemma}
\label{lem:transferring-invariance-conditions}
    Let $K \subset G$ be a finite subset with $e \in K$. For any two finite subsets $F_0$, $F_1 \subset G$, it holds that \[
        |KF_1 \setminus F_1| \leq |KF_0 \setminus F_0| + |K||F_0 \triangle F_1|.
    \]
\end{lemma}

Next, we review concepts relating to the geometry of finite subsets of $G$.

\begin{definition}[Boundary and interior]
\label{def:bdry-int}
    Let $F \subset G$ and $K \subset G$ be finite subsets. The \textit{$K$-boundary} of $F$ is the subset \[
        \partial_K F = \{f\in F : Kf \not\subset F\}
    \] and the \textit{$K$-interior} of $F$ is the subset \[
        \int_K F = \{f\in F : Kf \subset F\}.
    \] Observe $F = \partial_K F \sqcup \int_K F$.
\end{definition}

We shall need the following elementary lemmas; see \cite[Lemma~2.1, Lemma~2.2]{bland-mcgoff-pavlov}.

\begin{lemma}
\label{lem:intersect-int-implies-containment}
    Let $F$, $K \subset G$ be finite subsets and let $g\in G$ be arbitrary. If $Kg$ intersects $\int_{KK^{-1}}F$, then $Kg \subset F$.
\end{lemma}

\begin{lemma}
\label{lem:bdry-size-bound-for-invariance}
    Let $F$, $K \subset G$ be finite subsets. It holds that \[
        |\partial_K F| \leq |K| |KF \triangle F|.
    \]
\end{lemma}

Next we review the notion of \textit{density} for infinite subsets of $G$. Roughly speaking, an infinite subset $C \subset G$ has density $\rho \in [0,1]$ if, whenever a finite subset $F$ is sufficiently large, it holds for all $g\in G$ that \[
    |F\cap Cg| \sim \rho|F|.
\] 
We use the F\o lner sequence to make this notion precise.


\begin{definition}[Banach density]
\label{def:banach-density}
    Given a subset $C \subset G$, the \textit{upper Banach density} of $C$ is \[
        \overline{D}(C) = \liminf_{n\to\infty} \sup_{g\in G}\frac{|F_n\cap Cg|}{|F_n|}
    \] and the \textit{lower Banach density} of $C$ is \[
        \underline{D}(C) = \limsup_{n\to\infty} \inf_{g\in G} \frac{|F_n \cap Cg|}{|F_n|}.
    \]
\end{definition}

These definitions have also appeared in the recent work on quasi-tilings of amenable groups due to Downarowicz, Huczek, and Zhang \cite{downarowicz-huczek, downarowicz-huczek-zhang, downarowicz-zhang-symb-ext, downarowicz-zhang-comparison}. 
The value of the upper (resp. lower) density does not depend on the choice of F\o lner sequence \cite[Lemma~2.9]{downarowicz-huczek-zhang}. Note that $\overline{D}(C) = 1- \underline{D}(G\setminus C)$ holds for any subset $C\subset G$.

Next we review a notion regarding how an infinite subset $C \subset G$ may be distributed throughout the group $G$.

\begin{definition}[Separation]
\label{def:separated}
    Given a finite subset $L \subset G$, we say an infinite subset $C \subset G$ is \textit{$L$-separated} if $Lc_1 \cap Lc_2 = \varnothing$ for every distinct pair $c_1 \neq c_2 \in C$. 
\end{definition}

Note that if $C$ is $L$-separated, then so is $Cg$ for each fixed $g\in G$. Using this fact, one may easily check that if $C$ is $L$-separated then $\overline{D}(C) \leq 1/|L|$. The following lemma states something slightly stronger.

\begin{lemma}
\label{lem:approx-dens-of-L-disjoint-subs}
    Let $M$, $L\subset G$ be finite subsets with $e \in M \subset L$. For any nonempty finite subset $F \subset G$ and any $L$-separated subset $C \subset G$, we have that \[
        \frac{|F\cap MC|}{|F|} \leq \frac{|M|}{|L|} + |M| \frac{|\partial_L F|}{|F|} + |M| \frac{|M^{-1}F\setminus F|}{|F|}.
    \]
\end{lemma}

\begin{proof}
    Note that $C$ is also $M$-separated by inclusion. Let $C^\times = \{c\in C : F\cap Mc  \neq \varnothing\}$. Observe that $C^{\times}$ is finite, as $C^\times \subset M^{-1}F$. By the fact that $C$ is $M$-separated, we have \[
        |F\cap MC| \leq \sum_{c\in C^\times}|M| = |M||C^\times|.
    \] Now let $C^\circ = \{c\in C : Lc\subset F\}$. Observe that $C^\circ \subset C^\times$. We emphasize that $C^\times$ is given in terms of $M$, while $C^\circ$ is given in terms of $L$. By definition, we have $LC^\circ \subset F$, in which case \[
        |F| \geq |LC^\circ| = |L||C^\circ|
    \] where the equality is a consequence of the fact that $C$ is $L$-separated.  Hence $|C^\circ| \leq |F|/|L|$. 
    
    Let $c \in C^\times \setminus C^\circ$ be fixed, in which case $F\cap Mc \neq \varnothing$ and $Lc \not\subset F$.  Therefore, if $c\in F$ then $c \in \partial_L F$, while if $c \notin F$ then $c \in M^{-1}F\setminus F$. This demonstrates that \[
        C^\times \setminus C^\circ \subset (\partial_L F) \cup (M^{-1}F\setminus F).
    \] From this, we see that \begin{align*}
        |F\cap MC| &\leq |M||C^\times|\\
        &= |M||C^\circ| + |M||C^\times \setminus C^\circ|\\
        &\leq |M|\frac{|F|}{|L|} + |M||\partial_L F| + |M| |M^{-1}F\setminus F|.
    \end{align*} After dividing by $|F|$, we obtain the conclusion.
\end{proof}

With $M$ and $L$ fixed as in the above lemma, if one chooses $F = F_n$ and lets $n$ approach infinity, then one may easily see that $\overline{D}(MC) \leq |M|/|L|$ whenever $C\subset G$ is $L$-separated. 
However, what is especially significant for our purposes here is that the density of $MC$ can be estimated by sets $F$ which are sufficiently large with respect to $M$ and $L$ alone, and there is no dependence on $C$ other than the fact that $C$ is $L$-separated. If $C$ were an arbitrary subset satisfying $\overline{D}(C) < \varepsilon$, then that density is not in general approximated by finite subsets except for very large ones, depending on $C$.


\subsection{Shifts and subshifts}

In this section, we briefly review symbolic dynamics for amenable groups and state a few useful lemmas. For a more thorough introduction to symbolic dynamics and dynamics on amenable groups, see \cite{lind-marcus} and \cite{kerr-li}.

\begin{definition}[Labelings and patterns]
Let $\A$ be a finite alphabet of symbols, endowed with the discrete topology. A function $x : G \to \A$ is an \textit{$\A$-labeling} of $G$. The set of all such labelings is denoted $\A^G$, which is endowed with the product topology. Given a finite subset $F \subset G$, a function $p : G \to \A$ is called a \textit{pattern}, said to be of \textit{shape} $F$. The set of all patterns of shape $F$ is denoted $\A^F$, and the set of all patterns of any shape is denoted $\A^*$. 
\end{definition}

Given a point $x \in \A^G$ and a finite subset $F \subset G$, in this paper we shall take $x(F)$ to mean the restriction of $x$ to $F$, which is itself a pattern of shape $F$. This is normally denoted $x|_F \in \A^F$, but we raise the $F$ from the subscript for readability.

\begin{definition}[Shifts and subshifts]
\label{def:subshift}
    The group $G$ acts on $\A^G$ by way of translations; for each fixed $g \in G$ we have a homeomorphism $\sigma^g : \A^G \to \A^G$ given by $\sigma^g(x)(g_1) = x(g_1g)$ for every $g_1 \in G$ and $x \in X$. The action $\sigma = (\sigma^g)_g$ is called the \textit{shift action} of $G$ on $\A^G$, and together $(\A^G,\sigma)$ is a dynamical system called the \textit{full shift} over $\A$. A \textit{subshift} is a subset $X \subset \A^G$ which is $\sigma$-invariant and closed in the topology of $\A^G$. 
\end{definition}

A \textit{fixed point} is a point $x\in \A^G$ such that $x(g_1) = x(g_2)$ for every $g_1$, $g_2\in G$, in which case $X = \{x\}$ is a trivial subshift. We shall say a subshift $X$ is \textit{nontrivial} if it contains at least two points.

\begin{definition}[Patterns in subshifts]
    A pattern $p \in \A^F$ is said to \textit{appear} in a point $x \in \A^G$ {at} an element $g \in G$ if $\sigma^g(x)(F) = p$. The set of all patterns of shape $F$ appearing in any point of $X$ is denoted $\P(F,X) \subset \A^F$. The set of all patterns of any shape appearing in any point of $X$ is denoted $\P(X) \subset \A^*$.
\end{definition}

Given a subshift $X \subset \A^G$ and a (finite or infinite) collection of patterns $\F \subset \A^*$, one may construct a subshift $X_0 \subset X$ by expressly \textit{forbidding} the patterns in $\F$ from appearing in the points of $X$. That is, \[
    X_0 = \{x\in X : \text{no pattern from $\F$ appears in $x$}\}.
\] We denote the subshift $X_0$ by $\langle X \mid \F \rangle$. Every subshift may be realized in this form; indeed, $X = \langle \A^G \mid \A^* \setminus \P(X)\rangle$ holds for every subshift $X \subset \A^G$. 

Next we review some special classes of subshifts.

\begin{definition}[Strongly aperiodic subshifts]
\label{def:aperiodic}
    A point $x$ is \textit{aperiodic} if $\sigma^g(x) = x$ only when $g = e$. A subshift $X$ is \textit{strongly aperiodic} if every $x \in X$ is aperiodic. In other words, the action $\sigma$ is \textit{free} on $X$.
\end{definition}

\begin{definition}[Shifts of finite type]
    A subshift $Y \subset \A^G$ is a \textit{shift of finite type (SFT)} if there exists a finite collection of patterns $\F \subset \A^*$ such that $Y = \langle \A^G \mid \F \rangle$. For such a subshift, it is always possible to take $\F$ in the form $\A^K \setminus \P(K,Y)$ for some finite subset $K \subset G$. In this case, we say that $K$ \textit{witnesses} $Y$ as an SFT.
\end{definition}

We will need the following elementary lemma; see
\cite[Lemma~2.8]{bland-mcgoff-pavlov}.


\begin{lemma}
\label{lem:sft-excision}
    Let $Y$ be an SFT witnessed by $K \subset G$, let $y_1$, $y_2 \in Y$ be arbitrary points and let $F\subset G$ be a finite subset. If $y_1(\partial_{KK^{-1}}F) = y_2(\partial_{KK^{-1}}F)$, then the point $y$ given by $y(g) = y_1(g)$ if $g \in F$ and $y(g) = y_2(g)$ if $g\in G\setminus F$ is a point belonging to $Y$.
\end{lemma}

\begin{definition}[Strong irreducibility]
\label{def:strong-irreducibility}
    A subshift $Y \subset \A^G$ is \textit{strongly irreducible} if there exists a finite subset $K \subset G$ with $e\in K$ such that for any finite subsets $F_1$, $F_2 \subset G$ and allowed patterns $p_1 \in \P(F_1, Y)$ and $p_2 \in \P(F_2, Y)$, if $KF_1$ is disjoint from $F_2$ then there is a point $y \in Y$ such that $y(F_1) = p_1$ and $y(F_2) = p_2$. In this case, we say that $K$ \textit{witnesses} $Y$ as strongly irreducible or that $Y$ is a strongly irreducible subshift of \textit{parameter} $K$. 
\end{definition}

Note that strong irreducibility is preserved under taking factors (Definition~\ref{def:homomorphism}) and products (Definition~\ref{def:product-systems}) of subshifts. Strong irreducibility is equivalent to a seemingly stronger mixing condition, as the following lemma demonstrates. The proof relies on a standard compactness argument; we omit it for brevity.


\begin{lemma}
\label{lem:strong-irred-implies-patch-mixing}
    Suppose $Y$ is a strongly irreducible subshift of parameter $K \subset G$. For any finite subset $F \subset G$ and points $y_1$, $y_2 \in Y$, there is a point $y \in Y$ such that \[
        y(F) = y_1(F) \text{ and } y(G \setminus KF) = y_2(G \setminus KF).
    \]
\end{lemma}




Next we review topological entropy of subshifts.

\begin{definition}[Entropy]
\label{def:entropy} 
    The \textit{(topological) entropy} of a subshift $X$ is given by \[
        h(X) = \lim_{n\to\infty} h(F_n, X),
    \] where $h(F, X) = \frac{1}{|F|} \log |\P(F, X)|$ for each nonempty finite subset $F\subset G$ and $(F_n)_n$ is a F\o lner sequence for $G$.
\end{definition}

It is well known \cite[Theorem~4.38]{kerr-li} that the limit above exists and does not depend on the choice of F\o lner sequence.

Next we review homomorphisms between subshifts.

\begin{definition}[Homomorphisms]
\label{def:homomorphism}
    Let $\A_X$ and $\A_Y$ be finite alphabets and let $X \subset \A_X^G$ and $Y \subset \A_Y^G$ be subshifts. A \textit{homomorphism} between $X$ and $Y$ is a map $\phi : X \to Y$ that is both continuous and shift-commuting. By the Curtis-Lyndon-Hedlund theorem \cite[Theorem~1.8.1]{springer-book}, a map $\phi : X \to Y$ is a homomorphism if and only if there exists a finite subset $F \subset G$ and a function $\Phi : \P(F,X) \to \A_Y$ such that for every $g\in G$ and $x\in X$ it holds that \[
        \phi(x)(g) = \Phi(\sigma^g(x)(F)).
    \] If $\phi$ is surjective, then $\phi$ is said to be a \textit{factor map}, $X$ is said to \textit{factor onto} $Y$, and $Y$ is said to be a \textit{factor} of $X$. If $\phi$ is injective, then $\phi$ is said to be an \textit{embedding} and $X$ is said to \textit{embed into} $Y$.
\end{definition}

If $X$ embeds into $Y$, then $h(X) \leq h(Y)$. If $X$ factors onto $Y$, then $h(X) \geq h(Y)$. 

Next we review the primary way by which we join two subshifts together.

\begin{definition}[Product systems]
\label{def:product-systems}
    Let $\A$ and $\Lambda$ be finite alphabets and let $X \subset \A^G$ and $T \subset \Lambda^G$ be subshifts. The product system $X \times T$ equipped with the action $\varsigma$ given by $\varsigma^g(x,t) = (\sigma^g(x), \sigma^g(t))$ is isomorphic to a subshift over the alphabet $\A \times \Lambda$, with $(x,t)$ corresponding to the point $\tilde{x}$ such that $\tilde{x}(g) = (x(g),t(g))$ for each $g\in G$. Abusing notation, we regard $X\times T$ as a subshift of $(\A\times\Lambda)^G$.
\end{definition}

The following is standard; see \cite[Exercise~4.1.5]{lind-marcus} for the case where $G = \mathbb{Z}$.

\begin{lemma}
\label{lem:entropy-of-product-system}
    Let $X$ and $T$ be subshifts. Then $h(X \times T) = h(X) + h(T)$.
\end{lemma}

\subsection{Quasi-tilings}

In this section, we review quasi-tilings of amenable groups, which were originally introduced and studied by Ornstein and Weiss \cite{ornstein-weiss}. We also state a theorem of Downarowicz and Huczek which is essential for our main result. 

\begin{definition}[Quasi-tilings]
\label{def:quasi-tiling}
    Let $\S = \{S_1, \ldots, S_r\}$ be a finite collection of finite, nonempty subsets of $G$ which is ``shift-irreducible" in the sense that there is no pair of distinct subsets $S_1$, $S_2 \in \S$ and element $g\in G$ for which $S_1 g = S_2$. 
    We refer to these subsets as \textit{shapes}.
    A \textit{quasi-tiling} of $G$ over $\S$ is an assignment of each shape $S\in \S$ to a (generally infinite) subset $C_S \subset G$ (called the set of \textit{centers} for $S$) such that the sets $\{C_S : S\in \S\}$ are pairwise disjoint and the map $(S,c) \mapsto Sc$ is injective over $\{(S,c) : S\in \S \text{ and } c\in C_S\}$.
    
    A quasi-tiling of $G$ over $\S$ may be encoded as a point of the symbolic space $\Lambda(\S)^G$, where $\Lambda(\S) = \S \cup \{\varnothing\}$ is thought of as an alphabet of $r + 1$ symbols. A point $t \in \Lambda(\S)^G$ encodes the quasi-tiling when $t(c) = S$ if and only if $c\in C_S$ for each $S\in \S$ and $c\in G$, and $t(g) = \varnothing$ otherwise. Here we shall identify $t$ with the quasi-tiling it formally encodes. We shall write $C(t) = \{g\in G : t(g) \neq \varnothing\}$, 
    which is precisely the set $\bigcup_{S\in\S} C_S$. A \textit{tile} of a quasi-tiling $t$ is a subset of $G$ of the form $t(c)c$ where $c\in C(t)$.
\end{definition}

A quasi-tiling $t$ is \textit{disjoint} if $t(g_1)g_1 \cap t(g_2)g_2 = \varnothing$ for every $g_1 \neq g_2 \in G$. A quasi-tiling $t$ is said to \textit{cover} the group $G$ if $\bigcup_{g} t(g)g = G$. An \textit{exact} tiling is one which is both disjoint and covers $G$ (in other words, the tiles of $t$ form a partition of $G$). For most applications, it is not necessary that quasi-tilings be exactly disjoint or exactly covering; it is often sufficient to have a quasi-tiling whose tiles are ``nearly disjoint" and which ``nearly covers" $G$. The following definition formalizes the ``nearly disjoint" condition.

\begin{definition}[Retractions and $\varepsilon$-disjointness]
\label{def:epsilon-disjoint}
    Given a quasi-tiling $t$, a \textit{retraction} of $t$ is any quasi-tiling $\ret(t)$ (which is in general given over a different collection of shapes than $t$) such that $\ret(t)(g) \subset t(g)$ for every $g \in G$. Given $\varepsilon > 0$, a quasi-tiling $t$ is said to be \textit{$\varepsilon$-disjoint} if it has a disjoint retraction $\ret(t)$ such that, for every $c\in C(t)$, it holds that \[
        |t(c) \setminus \ret(t)(c)| < \varepsilon|t(c)|.
    \] In words, every tile of $t$ may be ``retracted" to a subset of proportion at least $1-\varepsilon$ such that the retracted subsets are all pairwise disjoint.
\end{definition}

The following definition formalizes the ``nearly covering" condition.

\begin{definition}[$\rho$-covering]
\label{def:rho-covering}
    Given $\rho \in (0,1)$, a quasi-tiling $t$ is \textit{$\rho$-covering} if \[
        \underline{D}\Big(\bigcup_{g\in G}t(g)g\Big) \geq \rho
    \] where $\underline{D}$ is the lower Banach density (Definition~\ref{def:banach-density}).
\end{definition}

We shall need the following elementary lemma; see \cite[Lemma~3.4]{downarowicz-huczek-zhang}.

\begin{lemma}
\label{lem:covering-of-retract}
    Let $\rho_0$, $\rho_1 \in (0,1)$ be fixed. Suppose $t_0$ is a $\rho_0$-covering quasi-tiling and suppose $t_1$ is a disjoint retraction of $t_0$ such that $|t_1(c)| \geq \rho_1|t_0(c)|$ for each $c\in C(t_0)$. Then $t_1$ is $\rho_0\rho_1$-covering.
\end{lemma}


The existence of arbitrarily nice quasi-tilings over countable amenable groups 
(i.e., with arbitrarily large shapes and arbitrarily good near-disjointness and near-covering properties) has been known in some form since 1987, due first to Ornstein and Weiss \cite[I.\textsection 2~Theorem~6]{ornstein-weiss}. This construction was sharpened in 2015 by Downarowicz, Huczek and Zhang who demonstrated that a countable amenable group exhibits an \textit{exact} tiling with arbitrarily large shapes; moreover, one can find a system of such tilings which has topological entropy zero \cite[Theorem~5.2]{downarowicz-huczek-zhang}.

For our purposes, we require not just that a system of arbitrarily nice quasi-tilings exists, but also that one may be obtained as a topological factor of a given subshift $X$. 
We have the following theorem of Downarowicz and Huczek \cite[Lemma~3.4]{downarowicz-huczek}. Not every property claimed here was stated in their theorem (here we state property (5) and the fact that the map $t\mapsto \ret(t)$ in (4) is a homomorphism), but a close reading of their proof reveals that it may be minorly modified to conclude this slightly stronger result. Here we provide a short argument which fills in the gaps, appealing to the construction in \cite{downarowicz-huczek} as required.

\begin{theorem}
\label{thm:dynamical-quasitilings}
    Let $X$ be a strongly aperiodic subshift, let $\varepsilon \in (0,1/3)$ be arbitrary, and suppose that $r \in \mathbb{N}$ satisfies $(1-\varepsilon/2)^r < \varepsilon$. For any $n_0 \in \mathbb{N}$ and finite subset $L \subset G$, there is a collection of shapes $\S = \{F_{n_1}, \ldots, F_{n_r}\}$ and a subshift $T \subset \Lambda(\S)^G$ such that \begin{enumerate}
        \item $n_0 < n_1 < \cdots < n_r$,
        \item there is a factor map $\T : X \to T$,
        \item every $t \in T$ is $(1-\varepsilon)$-covering,
        \item every $t \in T$ is $\varepsilon$-disjoint as witnessed by a continuous and shift-commuting retraction map $t\mapsto \ret(t)$, and
        \item for every $t \in T$, the set $C(t)$ is $L$-separated.
    \end{enumerate}
\end{theorem}

\begin{proof}
    In \cite{downarowicz-huczek}, $n_1$ is chosen as $n_0 + 1$ and $n_i$ is inductively chosen so that $F_{n_i}$ is $(F_{n_j}, \delta_j)$-invariant for every $j < i$, where $\delta_j > 0$ is specified in the construction. From this, we infer that $n_i - n_{i-1}$ may be arbitrarily large for each $i = 1, \ldots, r$. We therefore additionally assume that for each $i = 1, \ldots, r$, we have that \begin{enumerate}[(S1)]
        \item $|F_{n_i} \cap F_{n_i}\ell| \geq (1-\varepsilon)|F_{n_i}|$ for every $\ell \in L^{-1}L$, and 
        \item $F_{n_{i-1}} L^{-1} L \subset F_{n_i}$. 
    \end{enumerate} This is possible because the sequence $(F_n)_n$ was chosen to be both left and right F\o lner, to be ascending in $n$, and to satisfy $\bigcup_n F_n = G$ (Definition~\ref{def:folner-seq}).
    
    Let $S_i = F_{n_i}$ for each $i = 1, \ldots, r$ and let $\S = \{S_1, \ldots, S_r\}$. Our modifications to the choice of $\S$ preserve the invariance conditions assumed by \cite{downarowicz-huczek}. We now define and construct everything else as in \cite{downarowicz-huczek}, summarized below.
    
    

    Let $x \in X$ be fixed. In \cite{downarowicz-huczek}, the quasi-tiling $t = \T(x) \in \Lambda(\S)^G$ is constructed inductively, with the tiles of shape $S_r$ chosen first, then $S_{r-1}$, and so on, down to $S_1$. Consequently, for each $c\in C(t)$ there is a well-defined subset $C(t)_{<c} \subset C(t)$ which denotes the centers of all the tiles laid \textit{before} the tile at $c$ in the inductive process. Some tiles of the same shape are laid simultaneously, so the implied ordering given by $c_1 < c_2$ if $c_1 \in C(t)_{<c_2}$ is not {total}. But, tiles laid simultaneously in the construction of \cite{downarowicz-huczek} are necessarily disjoint. 
    
    This in hand, the retraction map $t\mapsto \ret(t)$ is given by \[
        \ret(t)(c) = t(c) \setminus \Big( \bigcup_{c_0 \in C(t)_{<c}} t(c_0)c_0\Big)c^{-1}
    \] for each $c \in C(t)$, and $\ret(t)(g) = \varnothing$ otherwise. In \cite{downarowicz-huczek} it is shown by induction that this is a disjoint retraction satisfying $|t(c)\setminus\ret(t)(c)| < \varepsilon|t(c)|$ for each $c\in C(t)$. Moreover, it is quick to check that the map $t\mapsto \ret(t)$ is continuous and shift-commuting, as 
    the elements $c_0 \in C(t)_{<c}$ with $t(c)c \cap t(c_0)c_0 \neq \varnothing$ are determined by $\sigma^c(t)(F)$ for some (possibly very large) finite subset $F \subset G$. The construction in \cite{downarowicz-huczek} also gives that $t$ is $(1-\varepsilon)$-covering. 
    
    Let $T = \T(X) \subset \Lambda(\S)^G$. From the observations of the previous paragraph, we see that properties (1) through (4) hold. For the theorem, it remains to check property (5): that $C(t)$ is $L$-separated for each $t \in T$.
    Let $t \in T$ be fixed and suppose to the contrary that $Lc_1 \cap Lc_2 \neq \varnothing$ for some distinct $c_1 \neq c_2 \in C(t)$, in which case $c_1c_2^{-1} \in L^{-1}L$. Let $t(c_1) = S_i$ and $t(c_2) = S_j$ and suppose without loss of generality that $i \leq j$. 
    
    If $i = j$ then let $S = S_i = S_j$, in which case $|Sc_1 \cap Sc_2| = |S\cap Sc_1c_2^{-1}| \geq (1-\varepsilon)|S|$ by the assumed condition (S1). Note that $R_1 = \ret(t)(c_1) \subset S$ and $R_2 = \ret(t)(c_2) \subset S$ are subsets such that $|S \setminus R_1| < \varepsilon|S|$ and $|S \setminus R_2| < \varepsilon|S|$ by construction. Moreover, by construction $R_1c_1 \cap R_2c_2 = \varnothing$, in which case $Sc_1 \cap Sc_2 \subset (Sc_1 \setminus R_1 c_1) \cup (Sc_2 \setminus R_2 c_2)$. This implies that \[
        |Sc_1 \cap Sc_2| \leq |Sc_1 \setminus R_1c_1| + |Sc_2 \setminus R_2 c_2| < 2\varepsilon|S|.
    \] We thus obtain $(1-\varepsilon)|S| \leq |Sc_1\cap Sc_2| < 2\varepsilon|S|$, which contradicts the assumption that $\varepsilon < 1/3$.
    
    In the second case, suppose $i < j$. Then, $c_1 c_2^{-1} \in L^{-1} L$ implies that $S_ic_1c_2^{-1} \subset S_iL^{-1}L \subset S_j$ by the assumed condition (S2). Hence, $t(c_1)c_1 \subset t(c_2)c_2$. This doesn't immediately contradict the $\varepsilon$-disjointness of $t$, as it may be the case that $t(c_1)c_1 \subset t(c_2)c_2 \setminus \ret(t)(c_2)c_2$. However, here we appeal to a property of the retraction map which is easily checked by induction. For each $c \in C(t)$, we have that \begin{equation*}\tag{R1}
        \bigcup_{c_0 \in C(t)_{<c}} t(c_0)c_0 = \bigcup_{c_0 \in C(t)_{<c}} \ret(t)(c_0)c_0.
    \end{equation*} Moreover, the assumption that $i < j$ implies that $c_2 \in C(t)_{<c_1}$, as the tiles are placed in order of largest to smallest. In that case, we have that \[
        \ret(t)(c_1)c_1 \subset t(c_1)c_1 \subset t(c_2)c_2 \subset \bigcup_{c_0 \in C(t)_{<c_1}} t(c_0)c_0.
    \] This, together with property (R1), implies that $\ret(t)(c_1)c_1$ intersects $\ret(t)(c_0)c_0$ for some $c_0 \in C(t)_{<c_1}$, which contradicts the disjointness of $\ret(t)$.
    
    
    
    This covers all cases, so we conclude that $C(t)$ must be $L$-separated for every $t \in T$. With this, we have verified all properties for the theorem not already proved in \cite{downarowicz-huczek}.
\end{proof}

One sees as a consequence of the above theorem that $X$ factors directly onto $\ret(T)$, a system of disjoint quasi-tilings with arbitrarily large shapes and near-covering of $G$. This is stated by Downarowicz and Huczek \cite[Corollary~3.5]{downarowicz-huczek}. However, one has to give up control of the number of tile shapes in exchange for perfect disjointness of the tiles. 

For our purposes, we shall make use of the intermediate factor $T$. Each $t \in T$ carries all of the information needed to construct a nice, disjoint quasi-tiling (by way of taking the retraction $\ret(t)$); we retain control of the ``density" of that information by controlling the number of tile shapes and distributing the centers of the tiles arbitrarily sparsely throughout the group.


\subsection{Comparison property}

In this section, we turn our attention to the case where $G$ has the comparison property. In short, with the comparison property one may demonstrate that, if $\varepsilon$ is sufficiently small and if the shapes in $\S$ are sufficiently large, then the subshift $T$ in Theorem~\ref{thm:dynamical-quasitilings} factors onto a system  $T_1$ of exact tilings.

For a thorough discussion of the comparison property for countable amenable groups and its consequences, see \cite{downarowicz-zhang-symb-ext}. Here, we only repeat that the class of groups with the comparison property includes all countable groups with no finitely generated subgroup of exponential growth, and it is still unknown whether there exists a countable amenable group without the comparison property.





The following theorem is a consequence of the main results of \cite[Proposition~4.3, Theorem~4.7]{downarowicz-zhang-comparison}, also appearing in \cite[Proposition~6.10, Theorem~6.12]{downarowicz-zhang-symb-ext}. We state the result in this form for convenience; this is the form in which we shall appeal to the comparison property later in our construction.

Given a subshift $T$, suppose we assign to each $t\in T$ a subset $G_t \subset G$. We can encode each subset $G_t$ by its indicator function $\chi_{G_t} \in \{0,1\}^G$. We say that the assignment $t \mapsto G_t$ is continuous and shift-commuting if the map $t\mapsto \chi_{G_t}$ is continuous and shift-commuting in the usual sense (Definition~\ref{def:homomorphism}). Equivalently, there is a finite subset $F \subset G$ and a collection of patterns $\mathcal{G} \subset \P(F,T)$ such that for every $g\in G$ and $t \in T$, we have $g\in G_t$ if and only if $\sigma^g(t)(F) \in \mathcal{G}$.

\begin{theorem}
\label{thm:comparison-injection}
    Suppose $G$ has the comparison property. Let $T$ be a subshift over $G$ and suppose for every $t \in T$ we have corresponding disjoint subsets $A_t$, $B_t\subset G$ such that \begin{enumerate}
        \item the assignments $t \mapsto A_t$, $B_t$ are continuous and shift commuting, and
        \item there exists an $\varepsilon > 0$ such that $\underline{D}(B_t) - \overline{D}(A_t) > \varepsilon$ for every $t\in T$.
    \end{enumerate}
    
    Then there is a family of injections $\phi_t : A_t \to B_t$ induced by a block code, in the sense that there is a finite subset $F \subset G$ and a function $\Phi : \P(F,T) \to F$ such that for every $t\in T$ and every $g\in A_t$ it holds that \[
        \phi_t(g) = \Phi(\sigma^g(t)(F))g.
    \]
\end{theorem}

The following theorem is implicitly proved in both \cite[Theorem~6.3]{downarowicz-zhang-comparison} and \cite[Theorem~7.5]{downarowicz-zhang-symb-ext}. The proof, a construction which we adapt in part of our Theorem~\ref{thm:main-result}, relies on the characterization of the comparison property given in the previous theorem. We appeal to each of these theorems later, utilizing Theorem~\ref{thm:comparison-injection} in Theorem~\ref{thm:main-result} and utilizing Theorem~\ref{thm:quasi-to-exact-tiling-factor} in Theorem~\ref{thm:finding-the-good-target-Y0}.

\begin{theorem}
\label{thm:quasi-to-exact-tiling-factor}
    Let $G$ be a countable amenable group with the comparison property. For every finite subset $K \subset G$ and $\varepsilon > 0$, there is a $\delta > 0$ such that, if $T$ is a system of disjoint, $(1-\delta)$-covering quasi-tilings with $(K,\delta)$-invariant shapes, then $T$ factors onto a system $T_1$ of exact tilings with $(K,\varepsilon)$-invariant shapes, by way of a factor map $\ex : T \to T_1$ such that $C(t) = C(\ex(t))$ for every $t\in T$ and $t(c) \subset \ex(t)(c)$ for every $c\in C(t)$.

\end{theorem}

\section{Theorems}



\subsection{Target system}
\label{sect:finding-Y_0}

In order to construct an embedding from a given subshift $X$ into a given subshift $Y$, it will first be necessary for our construction to pass to a subsystem $Y_0 \subset Y$ in a way that preserves most of the conditions on $Y$. In this section, we construct that subshift $Y_0$.

In the following definition, we borrow a phrase from functional analysis.

\begin{definition}
\label{def:separate-elements}
    Let $G$ be a discrete group. We say that a subshift $Y$ over $G$ \textit{separates} elements of $G$ if for every pair of distinct elements $g_1$, $g_2 \in G$, there is a point $y\in Y$ such that $y(g_1) \neq y(g_2)$. Because $Y$ is shift-invariant, this is true if and only if it holds that for each $g\neq e$, there is a point $y\in Y$ such that $y(e) \neq y(g)$.
    
\end{definition}

This condition is not invariant under topological conjugacy. However, a subshift $Y_1$ is conjugate to a subshift $Y$ which separates elements of $G$ if and only if there is a finite subset $F\subset G$ such that, for every $g\neq e$, there is a point $y_1 \in Y_1$ with $\sigma^g(y_1)(F) \neq y_1(F)$. The forward implication follows from the Curtis-Lyndon-Hedlund theorem, and the converse implication follows from passing to a higher block presentation of $Y$.



This condition is similar to the condition that 
$Y$ has no \textit{global period}, meaning that there is no element $g \neq e$ such that $\sigma^g(y) = y$ for every $y \in Y$ (equivalently, the shift action $\sigma$ is \textit{faithful} on $Y$). If a subshift $Y$ separates elements of $G$, then $Y$ necessarily has no global period. The converse holds when $G$ is abelian, but not in general. When $Y$ is strongly irreducible, we have the following partial converse.



\begin{lemma}
\label{lem:global-period-to-sep-elems}
    Let $G$ be a discrete group and let $Y$ be a nontrivial strongly irreducible subshift over $G$. If $Y$ has no global period, then $Y$ is conjugate to a subshift which separates elements of $G$.
\end{lemma}

\begin{proof}
    
    Let $K\subset G$ be a finite subset containing $e$ which witnesses the strong irreducibility of $Y$ and write $K = \{e,k_1,k_2,\ldots,k_n\}$. Because $Y$ has no global period, for each $i \leq n$ there exists a point $y_i \in Y$ and an element $g_i \in G$ such that $\sigma^{k_i}(y_i)(g_i) \neq y_i(g_i)$. 
    
    Let $F = \{e, g_1, g_2, \ldots, g_n\}$ and let $g\in G\setminus\{e\}$ be chosen arbitrarily. Because $Y$ is nontrivial and strongly irreducible, if $g \in G \setminus K$ then we may construct a point $y \in Y$ such that $y(e) \neq y(g)$, in which case $y(F) \neq \sigma^g(y)(F)$. If instead $g = k_i$ for some $i\leq n$, then the fact that $y_i(g_i) \neq \sigma^{k_i}(y_i)(g_i)$ implies that $y_i(F) \neq \sigma^g(y_i)(F)$. 

    We have shown that for every $g\neq e$ there is a point $y\in Y$ for which $y(F) \neq \sigma^g(y)(F)$, in which case $Y$ is conjugate to a subshift which separates elements of $G$ by the observation in the paragraph following Definition~\ref{def:separate-elements}, and the lemma is completed. 
    
    

\end{proof}

The previous lemma demonstrates that when $Y$ is strongly irreducible, we may use conjugacy to pass back and forth between the condition that $Y$ has no global period and the condition that $Y$ separates elements of $G$. Indeed, in the proof of Theorem~\ref{thm:main-result} we shall appeal to the fact that $Y$ has no global period in order to replace $Y$ with a conjugate subshift which separates elements of $G$.

It is not in general true that a strongly irreducible subshift automatically has no global period. However, as noted in the proof of the above lemma, if a subshift $Y$ is nontrivial and strongly irreducible as witnessed by $K \subset G$, then for each $g\notin K$ there is a point $y\in Y$ such that $y \neq \sigma^g(y)$. This implies that any element which is a global period of $Y$ must belong to $K$ (moreover, the subgroup of $G$ consisting of all global periods of $Y$ must be contained in $K$). In this sense, to assume that a strongly irreducible subshift $Y$ also has no global period only imposes finitely many additional conditions on $Y$.

From this, we also see that if $G$ is a group containing no element of finite order (as in a torsion-free group like $\mathbb{Z}^d$), then a nontrivial strongly irreducible subshift over $G$ automatically has no global period. In fact, it is not difficult to show that a strongly irreducible subshift over a group $G$ with no element of finite order must necessarily separate elements of $G$.


We now proceed with the main construction for this section. Given a strongly irreducible SFT $Y$ which separates elements of $G$ and has positive entropy, the following theorem produces a subsystem $Y_0 \subset Y$ which is also strongly irreducible, separates elements of $G$, and has entropy in any arbitrary subinterval of $[0,h(Y)]$. The construction presented below is a modification of the construction presented in \cite[Theorem~4.1]{bland-mcgoff-pavlov}. Here we invoke the comparison property in order to construct a strongly irreducible system of exact tilings of $G$.

\begin{theorem}
\label{thm:finding-the-good-target-Y0}
    Let $G$ be a countable amenable group with the comparison property, let $Y$ be a strongly irreducible SFT over $G$ which separates elements of $G$, and let $\tilde{Y} \subset Y$ be a subshift satisfying $h(\tilde{Y}) < h(Y)$.
    For every subinterval $(a,b) \subset [h(\tilde{Y}), h(Y)]$, there is a strongly irreducible subshift $Y_0$ which separates elements of $G$ and satisfies $\tilde{Y} \subset Y_0 \subset Y$ and $a < h(Y_0) < b$.
\end{theorem}

\begin{proof}
    Let $\A$ be a finite alphabet such that $Y \subset \A^G$.
    Let $K \subset G$ be a finite subset with $e \in K$ chosen to witness $Y$ as a strongly irreducible SFT. We shall abbreviate $\int^n F = \int_{K^n} F$ and $\partial^n F = \partial_{K^n} F$ for each natural $n \in \mathbb{N}$ and finite subset $F \subset G$ for the remainder of this proof. We shall also abbreviate $\partial^n p = p(\partial^n F)$ for each pattern $p$ of shape $F$.
    
    Choose $\varepsilon > 0$ such that \[
        \varepsilon < \min\Big(\frac{b-a}{3+\log 2 + 2\log|\A|}, \frac{b-a}{5+4\log|\A|}\Big).
    \]
    
    
    It is a theorem of Frisch and Tamuz \cite[Theorem~2.1]{frisch-tamuz} that for any $(K,\delta)$ there exists a strongly irreducible system of disjoint, $(1-\delta)$-covering quasi-tilings of $G$ whose every shape is $(K,\delta)$-invariant and whose entropy is less than $\delta$. This, in combination with Theorem~\ref{thm:quasi-to-exact-tiling-factor} and the fact that strong irreducibility is preserved under factor maps, implies the following. There exists a finite collection of shapes $\S$ and a strongly irreducible system $T \subset \Lambda(\S)^G$ of exact tilings of $G$ such that $h(T) < \varepsilon$ and every shape $S\in \S$ satisfies the following. \begin{enumerate}[(S1)]
        \item $K \subset \int^2 S$,
        \item $|S| > \varepsilon^{-1}$ and $2|S| < e^{\varepsilon|S|}$,
        \item $|\partial^2 S| < \varepsilon|S|$, and
        \item $h(S,\tilde{Y}) < h(\tilde{Y}) + \varepsilon$.
    \end{enumerate}
    
    
    For the majority of this proof, we operate primarily in the product system $Z_0 = Y \times T$. For each finite subset $F \subset G$ and pattern $p \in \P(F,Z_0)$, we shall write $p = (p^Y,p^T)$ where $p^Y \in \P(F,Y)$ and $p^T \in \P(F,T)$. If $F = S$ for some $S \in \S$, then we shall describe $p$ as a ``block".
    
    A block $b \in \P(S,Z_0)$ is called \textit{aligned} if $b^T(e) = S$ (note $e\in K$ and $K\subset S$ by (S1)). Note that if $b$ is aligned then $b^T(s) = \varnothing$ for each $s \in S\setminus\{e\}$, by the fact that every tiling $t\in T$ is disjoint. For a given subshift $Z \subset Z_0$, we denote the subcollection of all aligned blocks of shape $S$ allowed in $Z$ by \[
        \P^a(S,Z) \subset \P(S,Z) \subset (\A \times \Lambda(\S))^S
    \] where the superscript $a$ identifies the subcollection. 
    
    
    Let $\pi : Z_0 \to Y$ be the projection map defined by $\pi(y,t) = y$ for each $(y,t) \in Z_0$, which is a homomorphism. For each subshift $Z \subset Z_0$ and each fixed $z = (y,t) \in Z$, we have $y = \pi(z) \in \pi(Z)$ and $t \in T$, thus $Z \subset \pi(Z) \times T$. Consequently, for each subshift $Z \subset Z_0$ it holds that \[
        h(Z) \leq h(\pi(Z) \times T) = h(\pi(Z)) + h(T) < h(\pi(Z)) + \varepsilon
    \] where above we have used Lemma~\ref{lem:entropy-of-product-system} and the fact that $h(T) < \varepsilon$.
    
    
    
    Let $S \in \S$ be fixed. Here we choose and fix a collection of aligned blocks $\mathcal{W}(S) \subset \P^a(S,Z_0)$. We shall refer to these as ``witness patterns", because they will later allow us to demonstrate (``witness") for each $g \in G$ a point $z \in Z_0$ such that $z^Y(e) \neq z^Y(g)$. These witness patterns shall be of three types. 
    
    For the first type, let $s\in S\setminus\{e\}$ be arbitrary. Because $Y$ separates elements of $G$, there is a pattern $b^Y \in \P(S,Y)$ such that $b^Y(e) \neq b^Y(s)$. If $b^T \in \P(S,T)$ is given by $b^T(e) = S$ and $b^T(s) = \varnothing$ for each $s \in S\setminus\{e\}$, then $b = (b^Y, b^T)$ is an allowed block in $Z_0$ which satisfies $b^Y(e) \neq b^Y(s)$. For each $s \in S\setminus\{e\}$, pick and save one such block $b$ to the collection $\mathcal{W}(S)$. 
    
    For the second type, let $0 \in \A$ be a distinguished symbol of the alphabet. For every $y\in Y$, we can find a block $b \in \P(S,Z_0)$ such that $b(e) = (0,S)$ and $b^Y(\partial^2 S) = y(\partial^2 S)$. 
    This is because $Y$ is strongly irreducible of parameter $K$ and property (S1). For each pattern $y(\partial^2 S) \in \P(\partial^2 S, Y)$, pick and save one such block $b$ to the collection $\mathcal{W}(S)$. 
    
    For the third type, let $1 \in \A$ be a distinguished symbol different from $0$. For each $s \in S$, we can find a block $b \in \P(S,Y)$ such that $b^T(e) = S$ and $b^Y(s) = 1$. For each $s \in S$,  pick and save one such block $b$ to the collection $\mathcal{W}(S)$. This completes the description of the collection $\mathcal{W}(S)$. Note that $|\mathcal{W}(S)| \leq 2|S| + \A^{|\partial^2 S|} < e^{\varepsilon|S|} \cdot |\A|^{\varepsilon|S|}$ by properties (S2) and (S3). 
    
    
    
    We now construct a descending chain of subshifts $(Z_n)_n$ of $Z_0$ and claim that the subshift $Y_0$ desired for the theorem shall be given by $Y_0 = \pi(Z_n)$ for some $n \leq N$. Suppose for induction that $Z_n \subset Z_0$ has been constructed for $n \geq 0$. If there exists a shape $S_n \in \S$ and an aligned block $\beta_n \in \P^a(S_n, Z_n)$ such that \begin{enumerate}[(B1)]
        \item {$\beta_n^Y$ does not appear in $\tilde{Y}$} ($\beta_n^Y \notin \P(S_n, \tilde{Y})$),
        \item {$\beta_n$ is not one of the reserved witness patterns} ($\beta_n \notin \mathcal W(S_n)$), and
        \item there exists an aligned block $b \in \P^a(S_n, Z_n)$ with $b \neq \beta_n$ and $\partial^2 b = \partial^2 \beta_n$, 
    \end{enumerate} then let $Z_{n+1} = \langle Z_n \mid \beta_n\rangle$. If no such block exists for any shape, then the descending chain is finite in length and $Z_n$ is the terminal subshift.
    
    In fact, the chain \textit{must} be finite in length. This is because $Z_{n+1} \subset Z_n$ implies $\P(S,Z_{n+1}) \subset \P(S,Z_n)$ for each $S \in \S$, and $\P(S_n,Z_{n+1}) \sqcup \{\beta_n\} \subset \P(S_n, Z_n)$ for each $n \geq 0$. Hence we have that \[
        \sum_{S\in\S} |\P(S,Z_n)|
    \] is a nonnegative integer sequence which strictly decreases with $n$, and therefore must terminate. Let $N \geq 0$ be the index of the terminal subshift, and note by construction that for each $S \in \S$ and each aligned block $b \in \P^a(S, Z_N)$, either $b$ is uniquely determined by $\partial^2 b$, or $b \in \mathcal{W}(S)$, or $b^Y \in \P(S,\tilde{Y})$.
    
    
    We note the following intermediate lemma which shall be referenced multiple times in the remainder of this proof.
    
    \begin{lemma}
    \label{lem:forbidden-block-removed}
        For each $n < N$ and $(y,t) \in Z_n$, there is a point $(y^*,t) \in Z_{n+1}$ satisfying $(y,t)(g) = (y^*,t)(g)$ for every $g \notin \bigcup_c \int t(c)c$, where the union is taken over all $c \in G$ with $\sigma^c(y,t) = \beta_n$.
    \end{lemma}
    
    \begin{proof}
        
        By property (B3), there is a block $b \in \P^a(S_n, Z_n)$ such that $b(\partial^2 S_n) = \beta_n(\partial^2 S_n)$. In words, we simply replace every appearance of $\beta_n$ in $(y,t)$ with $b$ to construct the point $(y^*,t)$.
        
        Precisely, let $(c_k)_k$ enumerate the group elements $c$ for which $\sigma^c(y,t) = \beta_n$, which is necessarily a subset of the elements $c$ for which $t(c) = S_n$ because $\beta_n$ is aligned. Let $y^* \in \A^G$ be given by $y^*(g) = b^Y(gc_k^{-1})$ whenever $g \in S_nc_k$, and $y^*(g) = y(g)$ for every $g \notin \bigcup_k S_nc_k$. This point is well defined because $t$ is a disjoint tiling. 
        
        Because $Y$ is an SFT of parameter $K$ and by Lemma~\ref{lem:sft-excision}, we see that $y^*$ is an allowed point of $Y$. By construction, for every $S \in \S$ and $c \in G$ with $t(c) = S$, either $\sigma^c(y,t)(S) = \sigma^c(y^*,t)(S)$ or $\sigma^c(y,t) = \beta_n$ and $\sigma^c(y^*, t) = b$. Consequently, no forbidden block $\beta_i$ for any $i < n$ can appear in $(y^*,t)$, else that would force an appearance of $\beta_i$ in $(y,t)$, where $\beta_i$ is already forbidden. Moreover, the block $\beta_n$ cannot appear in $(y^*,t)$ by construction. Thus, $(y^*,t) \in Z_{n+1}$ and we are done.

    \end{proof}
    
    Now continuing the proof of Theorem~\ref{thm:finding-the-good-target-Y0}, we claim that for each $n$, the subshift $\pi(Z_n)$ is strongly irreducible, separates elements of $G$, and satisfies $\tilde{Y} \subset \pi(Z_n) \subset Y$. Moreover, we claim that $h(\pi(Z_n)) - h(\pi(Z_{n+1})) < b-a$ for each $n < N$, and $h(\pi(Z_N)) - h(\tilde{Y}) < b-a$. 
    
    
    For $\tilde{Y} \subset \pi(Z_n) \subset Y$, let $\tilde{y} \in \tilde{Y}$ and $t \in T$ be arbitrary, in which case $(\tilde{y},t) \in Z_0$. Note that, for each $n < N$, the block $\beta_n$ cannot appear in $(\tilde{y},t)$, else that would imply that $\beta_n^Y$ appears in $\tilde{y}$, contradicting property (B1). Thus, $(\tilde{y},t) \in Z_n$ for each $n \leq N$. This demonstrates that $\tilde{Y} \times T \subset Z_n$, and in particular $\tilde{Y} \subset \pi(Z_n) \subset Y$, for each $n \leq N$.

    For the strong irreducibility, note that $Z_0 = Y \times T$ is strongly irreducible because both $Y$ and $T$ are strongly irreducible. For induction, suppose $Z_n$ is strongly irreducible of parameter $K_n \subset G$ (with $e \in K_n$) for a fixed $n < N$. Let $U = \bigcup_{S\in \S} S$. We claim that $Z_{n+1}$ is strongly irreducible of parameter $UU^{-1} K_n UU^{-1}$.
    
    
    
    Indeed, let $(y_1,t_1)$, $(y_2,t_2) \in Z_{n+1}$ be arbitrary points, and let $F_1$, $F_2 \subset G$ be arbitrary finite subsets with $K_nUU^{-1} F_1 \cap UU^{-1} F_2 = \varnothing$. As $Z_{n+1} \subset Z_n$ and $Z_n$ is strongly irreducible of parameter $K_n$, there is a point $(y,t) \in Z_n$ with $(y,t)(UU^{-1}F_1) = (y_1,t_1)(UU^{-1}F_1)$ and $(y,t)(UU^{-1}F_2) = (y_2,t_2)(UU^{-1}F_2)$. 
    
    Now consider the forbidden block $\beta_n$ of shape $S_n \in \S$. Suppose $\sigma^g(y,t)(S_n) = \beta_n$ for some $g \in G$ with $S_ng \cap F_1 \neq \varnothing$. It follows that $g \in S_n^{-1} F_1$, hence $S_ng \subset S_nS_n^{-1} \subset UU^{-1}F_1$, hence $\sigma^g(y_1,t_1)(S_n) = \sigma^g(y,t)(S_n) = \beta_n$, contradicting the fact that $(y_1,t_1) \in Z_{n+1}$ and $\beta_n$ is forbidden in $Z_{n+1}$. From this observation (and by an identical argument for $F_2$), we see that if $\beta_n$ appears anywhere in $(y_0,t)$ then it does \textit{not} appear on any tile of $t$ which intersects either $F_1$ or $F_2$. 
    
    Let $(y^*,t) \in Z_{n+1}$ be the point delivered by Lemma~\ref{lem:forbidden-block-removed} as applied to $(y,t)$. From the previous paragraph, we have $(y^*,t)(F_i) = (y,t)(F_i) = (y_i,t_i)(F_i)$ for each $i = 1,2$. We conclude that $Z_{n+1}$ is strongly irreducible of parameter $UU^{-1}K_n UU^{-1}$, therefore completing the induction. As strong irreducibility is preserved under taking factors, it follows that $\pi(Z_n)$ is strongly irreducible for each $n \leq N$.

    To see that each $\pi(Z_n)$ separates elements of $G$, let $n \leq N$ and $g \neq e$ be fixed. We proceed by cases on $g$.
    
    If $g \in S$ for some $S \in \S$, then there is a witness block $b \in \mathcal{W}(S)$ such that $b^Y(e) \neq b^Y(g)$. Choose and fix $(y_0, t) \in Z_0$ such that $(y_0,t)(S) = b$ and apply Lemma~\ref{lem:forbidden-block-removed} at most $N$ times to produce the point $(y_n,t) \in Z_n$. Note $(y_n,t)(S) = (y_0,t)(S) = b$, as $b \neq \beta_n$ for every $n < N$ by property (B2). Consequently, $y_n \in \pi(Z_n)$ satisfies $y_n(e) = b^Y(e) \neq b^Y(g) = y_n(g)$.
    
    If $g \notin S$ for any $S \in \S$, then pick $S_0 \in \S$ arbitrarily and fix $t \in T$ with $t(e) = S_0$. 
    Because $t$ is an exact tiling, there is a unique $c \neq e$ such that $g\in t(c)c$. Let $S_1 = t(c)$ and write $g = s_1c$ for some $s_1 \in S_1 \in \S$.
    Recall $0$, $1 \in \A$ are two distinguished symbols determined in the construction of $\mathcal{W}(S)$. By construction, there is a witness block $b_1 \in \mathcal{W}(S_1)$ with $b_1^Y(s_1) = 1$. Pick any $y \in Y$ with $\sigma^c(y)(S_1) = b_1$. By construction, there is a witness block $b_0 \in \mathcal{W}(S_0)$ with $b_0^Y(e) = 0$ and $\partial^2 b_0^Y = y(\partial^2 S_0)$.  Because $Y$ is an SFT of parameter $K$ and by property (S1) and Lemma~\ref{lem:sft-excision}, there is a point $y^* \in Y$ given by $y^*(S_0) = b_0^Y$ and $y^*(g) = y(g)$ otherwise. In particular, $\sigma^c(y^*)(S_1) = b_1^Y$. Then, apply Lemma~\ref{lem:forbidden-block-removed} at most $N$ times to the initial point $(y^*,t) \in Z_0$, thus yielding $(y_n, t) \in Z_n$. Observe that $(y_n,t)(S_0) = (y^*,t)(S_0) = b_0$ and $\sigma^c(y_n,t)(S_1) = \sigma^c(y^*,t)(S_1) = b_1$, because $b_0 \neq \beta_n \neq b_1$ for every $n < N$ by property (B2). Consequently, $y_n \in \pi(Z_n)$ satisfies $y_n(e) = 0 \neq 1 = y_n(g)$. This finishes all cases and demonstrates that $\pi(Z_n)$ separates elements of $G$ for each $n \leq N$.
    
    
    The proofs that $h(\pi(Z_n)) - h(\pi(Z_{n+1})) < b-a$ for every $n < N$ and $h(\pi(Z_N)) - h(\tilde{Y}) < b-a$ are nearly identical to arguments appearing in \cite[Theorem~4.1]{bland-mcgoff-pavlov}. We proceed quickly through the argument here. Choose a finite subset $F\subset G$ such that $|h(Z_n) - h(F,Z_n)| < \varepsilon$ for every $n\leq N$, $|h(F,T) - h(T)| < \varepsilon$ (this implies in particular that $h(F,T) < 2\varepsilon$ by choice of $T$), and $|F \setminus U^{-1}F| < \varepsilon|F|$, where $U = \bigcup_{S\in \S} S$. For each $n \leq N$, let \[
        P(n) = \sum_{t(F)} \prod_c |\P^a(\int t(c), Z_n)|
    \] where the sum is taken over all patterns $t(F) \in \P(F,T)$ and the product is taken over all $c\in C(t) \cap U^{-1}F \cap F$. For each $n \leq N$, it holds that \[
        P(n) \leq |\P(F,Z_n)| \leq |\A|^{2\varepsilon|F|}\cdot P(n).
    \] The first inequality follows from the strong irreducibility of $Y$ in combination with Lemma~\ref{lem:forbidden-block-removed}. The latter inequality follows from projecting a pattern $p\in \P(F,Z_n)$ to the interiors of the tiles described by $p^T = t(F) \in \P(F,T)$ which are contained in $F$. This projection determines $p$ up to the portion of $F$ not covered by those tile interiors, a subset of $F$ of size at most $2\varepsilon|F|$ by the combination of the fact that the tiling is exact, the assumed invariance condition on $F$, and property (S3).
    
    Moreover, for each $n < N$ it holds that $P(n) \leq 2^{\varepsilon|F|}\cdot P(n+1)$. This follows from the fact that $|\P^a(\int S,Z_n)| - |\P^a(\int S, Z_{n+1})| \leq 1$  (at most one aligned block is removed as one passes from $Z_n$ to $Z_{n+1}$), in combination with the assumed invariance condition on $F$.
    
    The above, in combination with the assumed entropy estimating properties of $F$, implies that \[
        h(Z_n) < h(Z_{n+1}) + 2\varepsilon + \varepsilon\log 2 + 2\varepsilon\log|\A|.
    \] This, together with the fact that $h(Z_n) < h(\pi(Z_n)) +\varepsilon$ for each $n \leq N$ and the choice of $\varepsilon$, finally gives that $h(\pi(Z_n)) - h(\pi(Z_{n+1}) < b-a$ holds for each $n < N$.
    
    
    Next, consider the terminal subshift $Z_N$. Recall that for every shape $S\in \S$, every aligned block $b\in \P^a(S,Z_N)$ either belongs to $\mathcal{W}(S)$, is uniquely determined by $\partial^2 b$, or satisfies $b^Y \in \P(S,\tilde{Y})$ by construction. Recall also that $|\mathcal{W}(S)| \leq e^{\varepsilon|S|} \cdot |\A|^{\varepsilon|S|}$. This, together with property (S4), imply that for every shape $S\in \S$ it holds that \[
        |\P^a(\int S, Z_N)| \leq e^{h(\tilde{Y})|S|}\cdot e^{2\varepsilon|S|} \cdot |\A|^{2\varepsilon|S|}.
    \] This, together with the facts that $h(F,T) < 2\varepsilon$ and $|\P(F,Z_N)| \leq |\A|^{2\varepsilon|F|}\cdot P(N)$ argued earlier, gives us that \[
        |\P(F,Z_N)| \leq e^{h(\tilde{Y})|F|}\cdot e^{4\varepsilon|F|} \cdot |\A|^{4\varepsilon|F|}
    \] in which case it follows that $h(Z_N) < h(\tilde{Y}) + 5\varepsilon + 4\varepsilon\log|\A|$. This and our choice of $\varepsilon$ finally give that $h(\pi(Z_N)) - h(\tilde{Y}) < b-a$.
    
    We have demonstrated that $h(\pi(Z_n)) - h(\pi(Z_{n+1})) < b-a$ for every $n < N$ and $h(\pi(Z_N)) - h(\tilde{Y}) < b-a$. As $\pi(Z_0) = Y$, there must therefore exist at least one $n\leq N$ for which $\pi(Z_n)$ satisfies $a < h(\pi(Z_n)) < b$, thus completing the proof.
    
    
    
\end{proof}

\subsection{Marker patterns}

Let $Y$ be a subshift over $G$. In this section, we construct \textit{marker patterns} for $Y$. Marker patterns are patterns $m \in \P(Y)$ for which if $m$ appears in some point $y \in Y$, then there is not another appearance of $m$ in $y$ except at possibly an arbitrarily large displacement. Ideally, distinct appearances of a marker pattern appear on disjoint regions of $G$. In practice, there is potentially some overlap. Marker patterns were constructed for $G = \mathbb{Z}^d$ by Lightwood \cite[Lemma~6.3]{lightwood} in the case that $Y$ is an SFT with the uniform filling property which contains a point with a finite orbit. Here we generalize the construction to the case that $G$ is an arbitrary countable amenable group and $Y$ is a strongly irreducible SFT which separates elements of $G$ and has positive entropy. In particular, we do \textit{not} invoke the comparison property in this construction. This construction, in combination with Theorem~\ref{thm:finding-the-good-target-Y0}, provides the marker patterns for $Y$ which we need in the proof of our main result.


We construct a marker pattern $m$ here by first constructing an aperiodic point $y\in Y$, then taking $m$ to be the pattern of a large-enough shape appearing in $y$ at $e$. We construct aperiodic points as follows: begin by passing to a subsystem $Y_1\subset Y$ such that $h(Y_1) < h(Y)$, in which case one may find arbitrarily many patterns which do not appear in $Y_1$ but which do appear in $Y$. By beginning with a point in $Y_1$ and mixing in one of these ``forbidden" patterns from $Y$ (via the strong irreducibility of $Y$), we obtain a point which lacks all but possibly finitely many periods. Then, we pass to a subsystem $Y_0 \subset Y_1$ (which is also assumed to be strongly irreducible and to separate elements of $G$) from which an auxiliary pattern can be found which lacks precisely that finite set of periods. Mixing this via $Y_0$ into our point from before yields the desired aperiodic point.

\begin{theorem}
\label{thm:the-marker-patterns}
    Let $Y_0 \subset Y_1 \subset Y$ be subshifts and suppose that $Y_0$ and $Y$ are strongly irreducible, $Y_0$ separates elements of $G$, and $h(Y_1) < h(Y)$. For any $r \in \mathbb{N}$, there exists a shape $M \subset G$ such that for any fixed $y_0 \in Y_0$, there are patterns $m_1, \ldots, m_r \in \P(M,Y) \setminus \P(M, Y_1)$ and points $y_1, \ldots, y_r \in Y$ satisfying \begin{enumerate}
        \item $m_i$ appears in $y_i$ only at $e$,
        \item $m_i$ does not appear in $y_j$ for any $j \neq i$, and
        \item $y_i(g) = y_0(g)$ for every $g \notin M$
    \end{enumerate} for each $i = 1, \ldots, r$.
\end{theorem}

\begin{proof}
    Because $h(Y_1) < h(Y)$, there exists a finite subset $F \subset G$ with $e \in F$ such that $|\P(F,Y) \setminus \P(F,Y_1)| > r$. Choose and fix $r$ distinct patterns $a_1, \ldots, a_r \in \P(F,Y) \setminus \P(F,Y_1)$.
    
    Choose a finite subset $K \subset G$ with $e \in K$ to witness the strong irreducibility of $Y$ and $Y_0$. Let $\{e, g_1, \ldots, g_N\}$ be an enumeration of $F^{-1}KF$. Let $M_0 = KF$ and suppose for induction that $M_{n-1} \subset G$ has been constructed for a fixed $n \in [1,N]$. Because $G$ is infinite, by the pigeonhole principle there exists an element $h_n \in G$ such that $Kh_n \cup Kh_ng_n$ is disjoint from $M_{n-1}$. Let $M_n = M_{n-1} \sqcup (Kh_n \cup Kh_ng_n)$. We claim that $M = M_N$ is the subset desired for the theorem. Note that \[
        M = KF \sqcup \Big(\bigsqcup_{n=1}^{N} K\{h_n, h_ng_n\}\Big)
    \] by induction.
    
    Let $y_0 \in Y_0$ be arbitrary. Let $y_0^{(0)} = y_0$ and suppose for induction that $y_0^{(n-1)} \in Y_0$ has been constructed for a fixed $n \in [1,N]$. Because $Y_0$ separates elements of $G$, there exists a point $y^* \in Y_0$ with $y^*(h_n) \neq y^*(h_ng_n)$. Because $Y_0$ is strongly irreducible of parameter $K$, we may find a point $y_0^{(n)} \in Y_0$ such that $y_0^{(n)}(\{h_n,h_ng_n\}) = y^*(\{h_n,h_ng_n\})$ and $y_0^{(n)}(g) = y_0^{(n-1)}(g)$ for every $g \notin K\{h_n,h_ng_n\}$.
    
    By induction, the point $y_0^{(N)} \in Y_0$ satisfies $y_0^{(N)}(h_n) \neq y_0^{(N)}(h_ng_n)$ for each $n = 1, \ldots, N$ and $y_0^{(N)}(g) = y_0(g)$ for every $g \notin \bigcup_{n=1}^N K\{h_n,h_ng_n\}$. 
    
    For each $i = 1, \ldots, r$, because $Y$ is strongly irreducible of parameter $K$, we may find a point $y_i \in Y$ such that $y_i(F) = a_i$ and $y_i(g) = y^{(N)}(g)$ for each $g \notin KF$. Hence by construction, $y_i(h_n) \neq y_i(h_ng_n)$ for each $n = 1, \ldots, N$, and $y_i(g) = y_0(g)$ for each $g \notin M$. Finally, let $m_i = y_i(M)$. We claim that these are the patterns and points desired for the theorem. See Figure~\ref{fig:good_markers} for an illustration of the construction.
    
    \begin{figure}[hbt]
        \centering
        \includegraphics[width=0.5\textwidth]{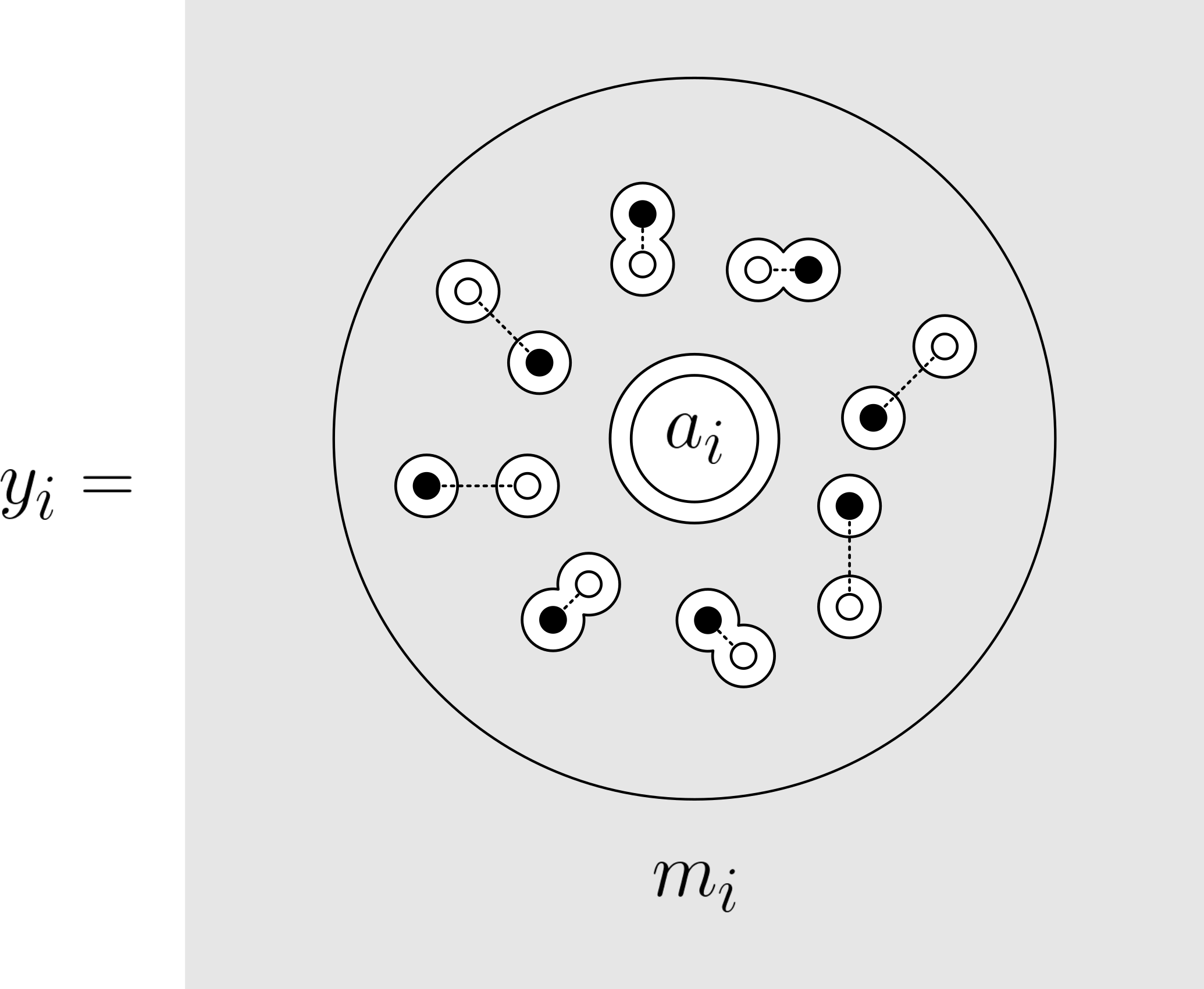}
        \caption{An illustration of the construction of the markers pattern $m_i$. The pattern $a_i$ is selected from $Y$ to be forbidden in $Y_1$. Thus, $y_i$ has no element of $G \setminus F^{-1}KF$ as a period. Then, small pairs of differing symbols are inductively mixed in (via $Y_0$) to prevent $y_i$ from having any period $g \in F^{-1}KF \setminus\{e\}$. The shaded exterior is the base point $y_0 \in Y_0$.}
        \label{fig:good_markers}
    \end{figure}

    For property (1), let $i \leq r$ be fixed. Observe that there is no $g \in F^{-1}KF\setminus\{e\}$ such that $\sigma^g(y_i)(M) = y_i(M)$, as otherwise would contradict the fact that every $g \in F^{-1}KF$ has a corresponding $h \in M \cap Mg^{-1}$ such that $y_i(h) \neq y_i(hg)$. 
    Moreover, if $g \notin F^{-1}KF$ then $Fg$ is disjoint from $KF$, in which case $\sigma^g(y_i)(F) = \sigma^g(y_0^{(N)})(F)$. We must therefore have $\sigma^g(y_i)(F) \neq a_i$, as otherwise would contradict the fact that $a_i$ is forbidden in $Y_1$ and hence also forbidden in $Y_0$. Consequently, $\sigma^g(y_i)(M) = y_i(M)$ only when $g = e$.
    
    For property (2), let $i \neq j$ be fixed. Observe that $a_j \neq a_i$ implies that $y_j(M) \neq y_i(M)$. Moreover, $\sigma^g(y_j)(M) \neq y_i(M)$ for any $g \neq e$ by an identical argument as in the previous paragraph. This is because $y_j(g) = y_0^{(N)}(g) = y_i(g)$ for every $g \notin KF$, and both $a_i$ and $a_j$ are forbidden in $Y_1$, and hence also forbidden in $Y_0$.
    
    Property (3) is true by construction.
\end{proof}

One may wonder why we mention the subshift $Y_1$ at all, instead of merely assuming $h(Y_0) < h(Y)$ directly and stating that the marker patterns are forbidden in $Y_0$. In fact, later on we shall have some additional structure on $Y_1$ which shall prove useful for our construction. Namely, in the chain of inclusions $Y_0 \subset Y_1 \subset Y$, we shall have that $Y_0$ is strongly irreducible, $Y_1$ is an SFT, and $Y$ is a strongly irreducible SFT. Then, having the marker patterns forbidden in not only $Y_0$ but also in $Y_1$ shall be significant.

\subsection{Main result}

We are now ready to present our main result. First we briefly outline the proof to come. 

Beginning with a strongly aperiodic subshift $X$, we derive a chain of factors of $X$ in the form \[
    X \xrightarrow{\T} T_0 \xrightarrow{\ret} T_1 \xrightarrow{\ex} T_2
\] where each $T_i$ for $i=0,1,2$ is a system of quasi-tilings of $G$. The foremost system $T_0$ is one in which the number of shapes is controlled, delivered by Theorem~\ref{thm:dynamical-quasitilings}. The system $T_1$ (consisting of retractions of points of $T_0$) is one in which each quasi-tiling exhibits large, disjoint tiles which nearly cover $G$. The lattermost system $T_2$ is a system of exact tilings, and we invoke the comparison property to construct it. Our construction of $T_2$ adapts part of the proof of \cite[Theorem~6.3]{downarowicz-zhang-comparison} (proof also appears in \cite[Theorem~7.5]{downarowicz-zhang-symb-ext}). We include the full construction here for the sake of completeness, and because we leverage our control over the foremost system $T_0$, which is not explicitly referenced in the constructions found in \cite{downarowicz-zhang-symb-ext, downarowicz-zhang-comparison}.

In principle, we wish to embed all the information of an arbitrary point $x\in X$ into a point $y\in Y$ in such a way that it can be uniquely (and ``locally") decoded. From $x$, we derive quasi-tilings $t_0$, $t_1$, and $t_2$ as above. On the side of $X$, the fact that $t_2$ is exact (and therefore \textit{covers} $G$) allows us to partition all of the information in $x$ into local ``blocks" of a bounded size. The fact that $t_2$ covers $G$ is especially significant for our construction; it would be fine if $t_2$ was not perfectly disjoint, so long as we had full covering and only a small (controllable) amount of ``redundancy". Then, by choosing our quasi-tilings with shapes sufficiently large, we may exhibit an injective map from patterns of shapes from $T_2$ in $X$ to patterns of shapes from $T_1$ in $Y$ (we describe this as a ``block injection"). Here we invoke the hypothesis $h(X) < h(Y)$, as well as the fact that each shape from $T_2$ is only slightly (controllably) larger in cardinality than a shape from $T_1$, as careful estimation shows.
We thereby construct a point of $Y$ by taking blocks from $x$, passing them through block injections to obtain blocks of $Y$, and gluing them together within $Y$ via a mixing condition. On this side, we see that the fact that $t_1$ is \textit{disjoint} is essential, so that there are no conflicts in laying these blocks together within $Y$. It is permissible for $t_1$ to not completely cover $G$, which is indeed the situation we grapple with in the coming proof.

That leaves open the issue of how to decode the point $x$ if only given $y$. If one knew which tilings $t_1$ and $t_2$ were used to construct $y$ from $x$, then it would be easy: simply look at the patterns appearing in $y$ on each of the tiles of $t_1$, pass these backwards through the block injections described above, then lay these new blocks upon the tiles of $t_2$ and thereby reconstruct $x$. Here we utilize the quasi-tiling $t_0$, which has a controlled number of shapes and tile centers spread arbitrarily sparsely throughout the group. This forces the ``information density" of $t_0$ to be arbitrarily low, hence we are able to encode $t_0$ within $y$ (with the use of marker patterns) by giving up only a subset of symbols of controllably small density. 

Given $y$, one is therefore able to decode the point $t_0$ by looking at the marker patterns, thereby deriving both $t_1$ and $t_2$, thereby decoding $x$ by the algorithm outlined above.



\begin{theorem}
\label{thm:main-result}
    Let $G$ be a countable amenable group with the comparison property. Let $X$ be a nonempty strongly aperiodic subshift over $G$. Let $Y$ be a strongly irreducible SFT over $G$ with no global period. 
    If $h(X) < h(Y)$ and $Y$ contains at least one factor of $X$, then $X$ embeds into $Y$.
\end{theorem}

\begin{proof}
    Let $\A_X$ and $\A_Y$ be finite alphabets such that $X\subset \A_X^G$ and $Y\subset \A_Y^G$. Suppose $\phi : X \to Y$ is a homomorphism, not necessarily injective, in which case $\tilde{Y} = \phi(X) \subset Y$ is a factor of $X$. Note that $h(\tilde{Y}) \leq h(X) < h(Y)$. 
    Without loss of generality (by Lemma~\ref{lem:global-period-to-sep-elems}), we may assume that $Y$ separates elements of $G$ (Definition~\ref{def:separate-elements}). In that case, by Theorem~\ref{thm:finding-the-good-target-Y0} there exists a strongly irreducible subshift $Y_0$ which separates elements of $G$ and satisfies $\tilde{Y} \subset Y_0 \subset Y$ and $h(X) < h(Y_0) < h(Y)$. By \cite[Theorem~4.2]{bland-mcgoff-pavlov}, there exists an SFT $Y_1$ such that $Y_0 \subset Y_1 \subset Y$ and $h(Y_1) < h(Y)$. It is significant for our proof that these inequalities are strict.
    
    Choose a finite subset $K \subset G$ with $e\in K$ such that $K^{-1} = K$, $K$ witnesses $Y$ and $Y_1$ as SFTs, and $K$ witnesses $Y$ and $Y_0$ as strongly irreducible. 
    As in the proof of Theorem~\ref{thm:finding-the-good-target-Y0}, we shall abbreviate $\int^n F = \int_{K^n} F$ and $\partial^n F = \partial_{K^n} F = F \setminus \int^n F$ for each natural $n \in \mathbb{N}$ and finite subset $F \subset G$ for the remainder of this proof. We shall also abbreviate $\partial^n p = p(\partial^n F)$ for each pattern $p$ of shape $F$.
    
    Choose $\varepsilon > 0$ such that $\varepsilon < 1/3$ and \[
        \varepsilon < \frac{h(Y_0) - h(X)}{2  + 5\log|\A_X| + (5+4|K|^6)\log |\A_Y|}.
    \] Choose $r = 1 + \lceil(2/\varepsilon)\log(1/\varepsilon)\rceil$, in which case $(1-\varepsilon/2)^r < \varepsilon$. Let $M \subset G$ be the subset delivered by Theorem~\ref{thm:the-marker-patterns} for $Y_0$, $Y_1$, $Y$, and $r$ as chosen here. By passing to a superset of $M$ if necessary, we may assume that $K \subset M$ and that $M^{-1} = M$. Choose a finite subset $L \subset G$ such that $M^6 \subset L$ and $|M^6|/|L| < \varepsilon$.
    
    Let $T_0 \subset \Lambda(\S_0)^G$ be the quasi-tiling system delivered by Theorem~\ref{thm:dynamical-quasitilings} for $X$, $\varepsilon$, $r$, and $L$ as chosen here, with $n_0$ chosen so that every $S_0 \in \S_0$ satisfies the following conditions. \begin{enumerate}[(S1)]
        \item $1/|S_0| < \varepsilon(1-\varepsilon)$.
        \item $\big(|K^3||M^6||L|\big)|LS_0 \setminus S_0| < \varepsilon|S_0|$.
        \item $h(S_0, X) < h(X) + \varepsilon$ and $h(S_0, Y_0) > h(Y_0) - \varepsilon$.
    \end{enumerate}
    
    Note that Theorem~\ref{thm:dynamical-quasitilings} gives that $|\S_0| \leq r$, so choose and fix an enumeration $(S_0^{(i)})_{i}$ of $\S_0$ where $i = 1, \ldots, r$. Theorem~\ref{thm:dynamical-quasitilings} also gives that every quasi-tiling $t_0 \in T_0$ is $\varepsilon$-disjoint as witnessed by a continuous and shift-commuting retraction map $t_0 \mapsto \ret(t_0)$. Let $\S_1$ be the set of all shapes realized as tiles of $\ret(t_0)$ for each $t_0 \in T_0$. Each shape $S_0 \in \S_0$ gives rise to a subcollection of shapes $S_1 \in \S_1$, all of which satisfy $S_1 \subset S_0$ and $|S_0 \setminus S_1| < \varepsilon|S_0|$. 
    Let $T_1 = \ret(T_0) \subset \Lambda(\S_1)^G$ be the system of all quasi-tilings obtained by taking retractions of the quasi-tilings in $T_0$, in which case the map $\ret$ is a factor map from $T_0$ to $T_1$. Note that every $t_1 \in T_1$ is disjoint by Theorem~\ref{thm:dynamical-quasitilings}. Moverover, every $t_1 \in T_1$ is $(1-\varepsilon)$-covering, since the retraction map {$\ret$} has the property that $\bigcup_g t_0(g)g = \bigcup_g \ret(t_0)(g)g$ for each $t_0 \in T_0$, and every $t_0 \in T_0$ is $(1-\varepsilon)$-covering by Theorem~\ref{thm:dynamical-quasitilings}.
    
    We aim to go one step further and construct a factor $T_2$ of $T_1$, which shall be a system of exact tilings. It is in this step that we appeal to the comparison property of $G$. To begin, note that for each pair of shapes $S_0 \in \S_0$ and $S_1 \in \S_1$ with $S_1 \subset S_0$ and $|S_0 \setminus S_1| < \varepsilon|S_0|$, it holds that \[
        |S_1| = |S_0| - |S_0 \setminus S_1| > (1-\varepsilon)|S_0| > 1/\varepsilon
    \] where above we have used property (S1). This implies that there exists an integer in the interval $[2\varepsilon|S_1|, 3\varepsilon|S_1|)$. Therefore, for each shape $S_1 \in \S_1$, we may find and fix an arbitrary subset $B(S_1) \subset S_1$ such that $2\varepsilon|S_1| \leq |B(S_1)| < 3\varepsilon|S_1|$. 
    
    To each $t_1 \in T_1$, we assign two disjoint subsets $A_{t_1}$, $B_{t_1}$ of $G$ in the following way. Let $A_{t_1} = G \setminus \bigcup_g {t_1}(g)g$ and let $B_{t_1} = \bigcup_g B({t_1}(g))g$. Observe that the assignments ${t_1} \mapsto A_{t_1}$, $B_{t_1}$ are continuous and shift-commuting. Observe also that $\overline{D}(A_{t_1}) \leq \varepsilon$ because ${t_1}$ is $(1-\varepsilon)$-covering, and that $\underline{D}(B_{t_1}) \geq 2\varepsilon(1-\varepsilon)$ by Lemma~\ref{lem:covering-of-retract} (with $\rho_0 = 1-\varepsilon$ and $\rho_1 = 2\varepsilon$). Consequently, \[
        \underline{D}(B_{t_1}) - \overline{D}(A_{t_1}) \geq 2\varepsilon(1-\varepsilon) - \varepsilon > 0
    \] where above we have used the fact that $\varepsilon \in (0,1/2)$. 
    
    Therefore, by Theorem~\ref{thm:comparison-injection}, there exists a family of injections $\phi_{t_1} : A_{t_1} \to B_{t_1}$ which is induced by a block code, in the sense that there is a finite subset $F \subset G$ and a function $\Phi : \P(F, T_1) \to F$ such that for every $t_1\in T_1$ and $g\in A_{t_1}$, it holds that \[
        \phi_{t_1}(g) = \Phi(\sigma^g({t_1})(F))g.
    \]
    
    With this, we are ready to construct $T_2$. For each $t_1 \in T_1$, let $t_2 = \ex(t_1)$ be the quasi-tiling obtained according to the rule \[
        t_2(c) = t_1(c) \sqcup \phi_{t_1}^{-1}(B(t_1(c))c)c^{-1} 
    \] for every $c \in C(t_1)$, and $t_2(g) = \varnothing$ otherwise. In words, we expand each tile $t_1(c)c$ by including all group elements which map into $B(t_1(c))c \subset B_{t_1}$ under the injective map $\phi_{t_1}$. Consequently, $|t_2(c) \setminus t_1(c)| \leq |B(t_1(c))| < 3\varepsilon|t_1(c)|$ for every $c \in C(t_2) = C(t_1)$.
    
    
    Let $\S_2$ be the set of all shapes realized as tiles of $\ex(t_1)$ for each $t_1 \in T_1$. There are at most finitely many because $\ex(t_1)(g) \subset F^{-1}t_1(g)$ for every $g\in G$. Each shape $S_1 \in \S_1$ gives rise to a subcollection of shapes $S_2 \in \S_2$, all of which satisfy $S_1 \subset S_2$ and $|S_2 \setminus S_1|< 3\varepsilon|S_1|$. 
    
    Let $T_2 = \ex(T_1) \subset \Lambda(\S_2)^G$. We claim that the map $\ex$ is a homomorphism and that every $t_2 \in T_2$ is an exact tiling of $G$. To see that the map $\ex$ is continuous and shift-commuting, let $t_1 \in T_1$ be arbitrary and let $t_2 = \ex(t_1)$. We note that for each $g \in G$ and $c \in C(t_2) = C(t_1)$, the definition given above implies that $g\in t_2(c)$ if and only if \[
        g\in t_1(c) \text{ or } \Phi(\sigma^{gc}(t_1)(F))g \in B(t_1(c)).
    \] Moreover, for each $f \in F$ we have $\sigma^{gc}(t_1)(f) = t_1(fgc) = \sigma^c(t_1)(fg)$. Because $t_2(c) \subset F^{-1}t_1(c)$, we see that $t_2(c)$ depends only on $\sigma^c(t_1)(FF^{-1}U_1)$, where $U_1 = \bigcup_{S_1\in \S_1} S_1$. This demonstrates that the map $\ex$ is induced by a block code, which is sufficient to demonstrate that $\ex$ is continuous and shift-commuting.
    
    For the claimed exactness, let $t_2 \in T_2$ and $g\in G$ be arbitrary. Choose $t_1 \in T_1$ such that $t_2 = \ex(t_1)$. If there exists a $c \in C(t_1)$ such that $g\in t_1(c)c$, then the $c$ is necessarily unique by the disjointness of $t_1$ and also it follows that $g \in t_2(c)c$. Otherwise, $g\in A_{t_1}$, in which case $\phi_{t_1}(g) \in B_{t_1} = \bigcup_c B(t_1(c))c$. Therefore there exists a $c \in C(t_1)$ such that $\phi_{t_1}(g)\in B(t_1(c))c \subset t_1(c)c$. The $c$ must again be unique by the disjointness of $t_1$. Then $g\in \phi_{t_1}^{-1}(B(t_1(c)c)) \subset t_2(c)c$. We see that for each $g\in G$ there exists a unique $c \in C(t_2)$ such that $g\in t_2(c)c$, hence $t_2$ is an exact tiling of $G$.
    
    There is one last quasi-tiling system we shall need. Let $\S_1^*$ denote the collection of all shapes obtained in the form $\int^3(S_1 \setminus M^6C)$, where $C$ is an arbitrary $L$-separated subset of $G$ and $S_1 \in \S_1$.  Each shape $S_1 \in \S_1$ gives rise to a subcollection of shapes $S_1^* \in \S_1^*$, all of which satisfy $S_1^* \subset S_1$. 
    For each $t_1 \in T_1$, let $t_1^*$ be the retraction of $t_1$ such that, for each $c\in C(t_1)$, it holds that \[
        t_1^*(c) = \int^3(t_1(c) \setminus M^6C(t_1)c^{-1})
    \] and $t_1^*(g) = \varnothing$ otherwise. Observe that each such quasi-tiling belongs to $\Lambda(S_1^*)^G$. Let $T_1^* = \{t_1^* \in \Lambda(\S_1^*)^G : t_1 \in T_1\}$. It is quick to see that the map $t_1 \mapsto t_1^*$ is a factor map from $T_1$ to $T_1^*$, because the assignment $t_1 \mapsto C(t_1)$ is continuous and shift-commuting. Moreover, $C(t_1^*) = C(t_2) = C(t_1) = C(t_0)$, thus all of these subsets are $L$-separated.
    
    With our quasi-tiling systems constructed, we now aim to construct an injection which will carry patterns on tiles in $X$ to patterns on tiles in $Y_0$. On the side of $X$, we use patterns of those shapes belonging to $\S_2$. On the side of $Y_0$, we shall use patterns of those shapes belonging to $\S_1^*$. We separate out the construction (and its rather involved estimations) into the following lemma. 
    
    \begin{lemma}
    \label{lem:block-injection}
        Let $(S_0, S_1, S_2, S_1^*)$ be any tuple of shapes from $\S_0$, $\S_1$, $\S_2$, and $\S_1^*$ respectively such that $S_1^* \subset S_1 \subset S_0 \cap S_2$, $|S_0 \setminus S_1| < \varepsilon|S_0|$, $|S_2 \setminus S_1| < 3\varepsilon|S_1|$, and $S_1^* = \int^3(S_1 \setminus M^6C)$ for some $L$-separated subset $C\subset G$. Then \[
            |\P(S_2, X)| < |\P(S_1^*, Y_0)|.
        \]
    \end{lemma}
    
    \begin{proof}
        We begin on the side of $X$. By the hypotheses, a quick calculation shows that $|S_0 \triangle S_2| < 5\varepsilon|S_0|$. It therefore holds that \begin{equation}\tag{X1}
        \label{ineq:upper-bound-on-X}
            |\P(S_2, X)| \leq |\P(S_0, X)| \cdot |\A_X|^{|S_0 \triangle S_2|} < e^{(h(X)+\varepsilon)|S_0|} \cdot |\A_X|^{5\varepsilon|S_0|}
        \end{equation} where above we have used condition (S3) in the latter inequality.
        
        Now we shall estimate the number of patterns in $Y_0$ of shape $S_1^*$. We begin by estimating the size of $S_1^*$ in terms of the size of $S_0$. We proceed in stages, beginning with $S_1$. We first note that \begin{align*}
            |S_1 \cap M^6C| & \leq |S_0 \cap M^6C|\\
            &\leq \frac{|M^6|}{|L|}|S_0| + |M^6||\partial_L S_0| + |M^6||(M^6)^{-1}S_0 \setminus S_0|\\
            &< \varepsilon|S_0| + \varepsilon|S_0| + \varepsilon|S_0|\\
            & = 3\varepsilon|S_0|
        \end{align*} where above we have used the fact that $S_1 \subset S_0$, Lemma~\ref{lem:approx-dens-of-L-disjoint-subs}, the choice of $L$, Lemma~\ref{lem:bdry-size-bound-for-invariance}, and the assumed condition (S2) in conjunction with the fact that $(M^6)^{-1} = M^6 \subset L$. It follows that \[
            |S_1 \setminus M^6C| = |S_1| - |S_1 \cap M^6C| \geq (1-\varepsilon)|S_0| - 3\varepsilon|S_0| = (1-4\varepsilon)|S_0|.
        \]
        
        This, together with the fact that $S_1 \setminus M^6C \subset S_1 \subset S_0$, gives us \begin{align*}
            |\partial^3(S_1 \setminus M^6C)| &\leq |K^3||K^3(S_1 \setminus M^6C) \setminus (S_1\setminus M^6C)|\\
            &\leq |K^3||K^3S_0 \setminus S_0| + |K^3|^2|S_0 \setminus(S_1 \setminus M^6C)|\\
            &< \varepsilon|S_0| + |K|^6 \cdot 4\varepsilon|S_0|\\
            &= \varepsilon(1 + 4|K|^6)|S_0|
        \end{align*} where above we have used Lemma~\ref{lem:bdry-size-bound-for-invariance}, Lemma~\ref{lem:transferring-invariance-conditions}, the assumed condition (S2) in conjuction with the fact that $K^3 \subset L$, and the calculation of the previous display. 
        
        The combination of the previous two displays gives us \begin{align*}
            |\int^3(S_1 \setminus M^6C)| &= |S_1 \setminus M^6C| - |\partial^3(S_1 \setminus M^6C)| \\
            & \geq (1-4\varepsilon)|S_0| - \varepsilon(1+4|K|^6)|S_0| \\
            &= |S_0| - \varepsilon(5+4|K|^6)|S_0|.
        \end{align*}
        
        Recall that $S_1^* = \int^3(S_1 \setminus M^6C)$ and note that $|\P(S_0, Y_0)| \leq |\P(S_1^*, Y_0)| \cdot |\A_Y|^{|S_0 \setminus S_1^*|}$. This, in combination with the previous display, gives that \begin{align*}
            |\P(S_1^*, Y_0)| &\geq |\P(S_0, Y_0)| \cdot |\A_Y|^{-|S_0 \setminus S_1^*|}\\
            &> e^{(h(Y_0) - \varepsilon)|S_0|} \cdot |\A_Y|^{-\varepsilon(5+4|K|^6)|S_0|}
        \end{align*} where above we have used the assumed condition (S3). Our choice of $\varepsilon$ gives us \[
            h(X) + \varepsilon + 5\varepsilon\log|\A_X| < h(Y_0) - \varepsilon - \varepsilon(5+4|K|^6)\log|\A_Y|
        \] in which case the previous calculation, in combination with (\ref{ineq:upper-bound-on-X}), finally gives that \[
            |\P(S_2, X)| < |\P(S_1^*, Y_0)|.
        \]
    \end{proof}
    
    The previous lemma implies that, for each tuple $(S_0, S_1, S_2, S_1^*)$ satisfying the hypotheses of the previous lemma, there exists an injective map \[
        \Psi(S_2, S_1^*) : \P(S_2, X) \to \P(S_1^*, Y_0).
    \] Let these injections be chosen and fixed. We now continue the proof of Theorem~\ref{thm:main-result}.
    
    Here we choose our marker patterns. Pick a point $y^{(m)} \in Y_0$ arbitrarily to serve as the ``substrate" of the marker patterns. Let $y_1,\ldots, y_r \in Y$ and $m_i = y_i(M) \in \P(M,Y)\setminus \P(M,Y_1)$ be the points and patterns delivered by Theorem~\ref{thm:the-marker-patterns} for the choice of substrate $y^{(m)}$.
    
    Now we appeal to the strong irreducibility of $Y_0$ in order to construct certain ``mixing" patterns. We do this twice: once for patterns of shape $M^6$, and once for patterns of each shape $S_1^*\in \S_1^*$.
    
    Note that $K^2\cdot KM^3 \subset M^6$ because $K\subset M$, thus $KM^3 \subset\int^2(M^6)$ and thus $KM^3$ is disjoint from $\partial^2M^6$. The subshift $Y_0$ is strongly irreducible as witnessed by $K$, therefore for any pair of patterns $u\in\P(M^3, Y_0)$ and $\partial^2 v \in \P(\partial^2M^6, Y_0)$, there is a pattern $w\in \P(M^6,Y_0)$ such that $w(g) = u(g)$ for each $g\in M^3$ and $w(g) = \partial^2v(g)$ for each $g\in \partial^2M^6$. Choose and fix one such pattern $w$ for each choice of $u$ and $\partial^2v$; we shall denote it by $w = u\cup \partial^2 v$.

    Let $S_1^* \in \S_1^*$ be arbitrary and suppose $S_1^* = \int^3(S_1 \setminus M^6C)$ for some shape $S_1 \in \S_1$ and some $L$-separated subset $C \subset G$. Observe that $K^2\cdot KS_1^* \subset S_1\setminus M^6C$, thus $KS_1^*$ is disjoint from $\partial^2(S_1\setminus M^6C)$. Therefore, for each pair of patterns $u\in \P(S_1^*, Y_0)$ and $\partial^2 v\in \P(\partial^2(S_1\setminus M^6C), Y_0)$, there is a pattern $w \in \P(S_1\setminus M^6C, Y_0)$ such that $w(g) = u(g)$ for each $g \in S_1^*$ and $w(g) = \partial^2v(g)$ for each $g\in \partial^2(S_1\setminus M^6C)$. Choose and fix one such pattern $w$ for each choice of $u$ and $\partial^2v$; we shall again denote it by $w = u \cup \partial^2 v$.
    
    We are now ready to construct the map $\psi : X \to Y$ desired for the theorem. To begin, let $x\in X$ be fixed, let $t_0 = \T(x)\in T_0$, let $t_1 = \ret(t_0) \in T_1$, let $t_2 = \ex(t_1) \in T_2$, let $t_1^* \in T_1^*$ be delivered by $t_1$, and let $y_0 = \phi(x)\in Y_0$ be the point delivered by the homomorphism $\phi : X \to Y$ assumed to exist at the beginning of the proof. Let $C = C(t_0) = C(t_1) = C(t_2) = C(t_1^*) \subset G$.
    
    We construct the point $y = \psi(x)\in Y$ in two stages, first by constructing a point $y_1 \in Y_1$ which locally looks like a point of $Y_0$, and then modifying $y_1$ into a point of $Y$ by placing down the marker patterns.
    
    Note that for each $c \in C$, the tuple of shapes $(t_0(c), t_1(c), t_2(c), t_1^*(c))$ satisfies the hypotheses of Lemma~\ref{lem:block-injection} by construction. Let $y_1$ be the point satisfying the following two conditions for every $c\in C$ when $(S_1, S_2, S_1^*) = (t_1(c), t_2(c), t_1^*(c))$. \begin{equation*}\tag{C1}
    \label{eq:C1}
        \sigma^c(y_1)(M^6) = y^{(m)}(M^3) \cup \sigma^c(y_0)(\partial^2M^6)
    \end{equation*} \begin{equation*}\tag{C2}
    \label{eq:C2}
        \sigma^c(y_1)(S_1\setminus M^6Cc^{-1}) = \Psi(S_2, S_1^*)(\sigma^c(x)(S_2)) \cup \sigma^c(y_0)(\partial^2(S_1\setminus M^6C))
    \end{equation*} Everywhere else, let $y_1(g) = y_0(g)$. Here we utilize the block injection(s) $\Psi(S_2,S_1^*)$ delivered by Lemma~\ref{lem:block-injection}, as well as the mixing boundary patterns $w = u\cup \partial^2 v$ constructed earlier. The construction of $y_1$ is illustrated in Figure~\ref{fig:construction}.
    
    \begin{figure}[hbt]
        \centering
        \includegraphics[width=0.8\textwidth]{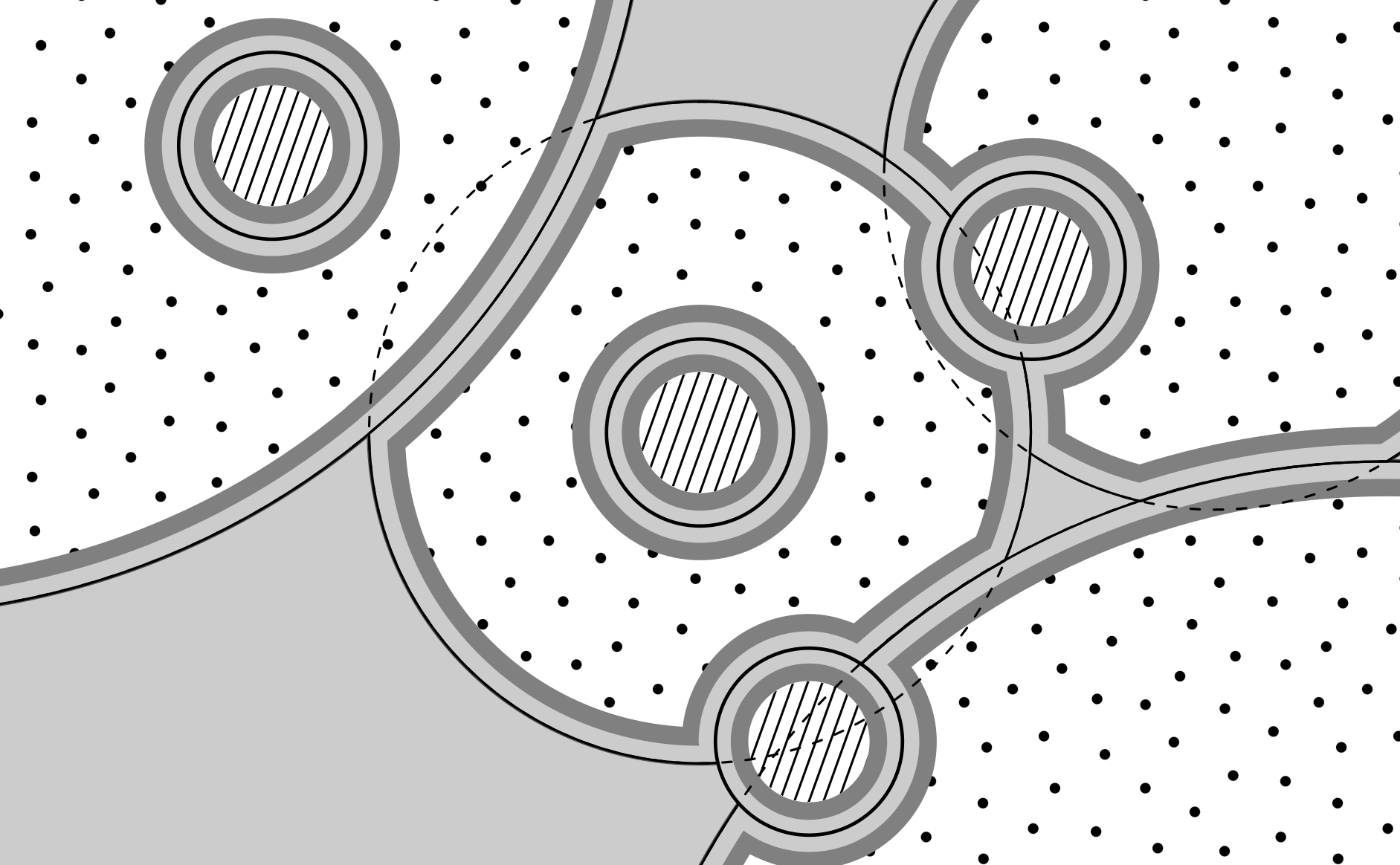}
        \caption{An illustration of the construction of $y_1$ for a hypothetical tiling of $\mathbb{Z}^2$ using different sized circles. The largest solid circles indicate the tiles of $t_0$ (the overlap, indicated by dashed lines, is removed in $t_1$). The smaller solid circles indicate the translates of $M^6$, wherein the marker patterns will later be placed. The stippled tile interiors are the patterns given by the block injections from Lemma~\ref{lem:block-injection}. The darkened boundaries are delivered by the strong irreducibility of $Y_0$, to mix each tile interior pattern with its respective boundary pattern from $y_0$. The small hatched circles are each labeled with the pattern drawn from the marker substrate, $y^{(m)}(M^3)$. The shaded exterior is the base point $y_0$.}
        \label{fig:construction}
    \end{figure}

    First we argue that $y_1$ is well-defined. Let $g\in G$ be fixed; we split over three cases. If $g\in M^6C$, then there is a unique $c \in C$ such that $g\in M^6c$ because $C$ is $L$-separated and $M^6 \subset L$. In this case, $y_1(g)$ is determined by condition (\ref{eq:C1}). If $g\in \Big(\bigcup_c t_1(c)c\Big)\setminus M^6C$, then again there is a unique $c\in C$ such that $g\in t_1(c)c\setminus M^6C = (t_1(c)\setminus M^6Cc^{-1})c$ by the disjointness of $t_1$. In this case, $y_1(g)$ is determined by condition (\ref{eq:C2}). Otherwise, $y_1(g) = y_0(g)$ is again uniquely determined.
    
    Next we argue that $y_1$ belongs to $Y_1$. Recall that $Y_1$ is an SFT as witnessed by $K$. Let $g\in G$ be arbitrary and consider the translate $Kg$. If $Kg$ intersects $\int^2(M^6)c$ for some $c\in C$, then $Kg \subset M^6c$ by Lemma~\ref{lem:intersect-int-implies-containment} and the fact that $K = K^{-1}$. Moreover, the $c$ must be unique. In this case, $\sigma^g(y_1)(K)$ is given by condition (\ref{eq:C1}) and therefore $\sigma^g(y_1)(K) \in \P(K,Y_0)$. If $Kg$ intersects $\int^2(t_1(c)c \setminus M^6C)$ for some $c\in C$, then $Kg \subset t_1(c)c \setminus M^6C$ also by Lemma~\ref{lem:intersect-int-implies-containment}, in which case the $c$ must again be unique. In this case, $\sigma^g(y_1)(K)$ is given by condition (\ref{eq:C2}) and therefore $\sigma^g(y_1)(K) \in \P(K,Y_0)$ again. If neither of these cases hold, then \[
        Kg \subset (G \setminus M^6 C) \cup \partial^2(M^6)C \cup \Big(G \setminus \bigcup_c (t_1(c)c \setminus M^6C)\Big) \cup \bigcup_c \partial^2(t_1(c)c \setminus M^6C).
    \] In this case, $\sigma^g(y_1)(K) = \sigma^g(y_0)(K)$ and therefore $\sigma^g(y_1)(K) \in \P(K,Y_0)$ again. We see that every pattern of shape $K$ appearing in $y_1$ is allowed in $Y_0$. Since $Y_0 \subset Y_1$ and $Y_1$ is an SFT witnessed by $K$, we conclude that $y_1 \in Y_1$.
    
    
    
    
    
    Next we argue that the map $x\mapsto y_1$ is continuous and shift-commuting. For a fixed $g \in G$, in order to determine the symbol $y_1(g)$, one must know first the tiles from $t_0$, $t_1$, $t_2$, and $t_1^*$ to which $g$ belongs. This requires looking only at the symbols of the involved quasi-tilings within a finite neighborhood of $g$. Then, one either applies condition (\ref{eq:C1}) or condition (\ref{eq:C2}) or returns $y_0(g)$. As every involved quasi-tiling and $y_0$ are derived from $x$ in a continuous and shift-commuting manner, it is evident that $y_1(g)$ depends only on $\sigma^g(x)(F_1)$, where $F_1$ is some (possibly very large but) finite subset of $G$. We conclude that the map $x \mapsto y_1$ is continuous and shift-commuting. 
    
    Finally, we construct $y = \psi(x) \in Y$ from $y_1$ as follows. For every $c\in C$, let $y$ satisfy $\sigma^c(y)(M) = m_i$, where $i\in [1,r]$ is the index of the shape $t_0(c) = S_0^{(i)} \in \S_0$ which was fixed at the beginning of the proof, and $m_i$ is the corresponding marker pattern. Everywhere else, let $y(g) = y_1(g)$.
    
    \begin{figure}[hbt]
        \centering
        \includegraphics[width=0.3\textwidth]{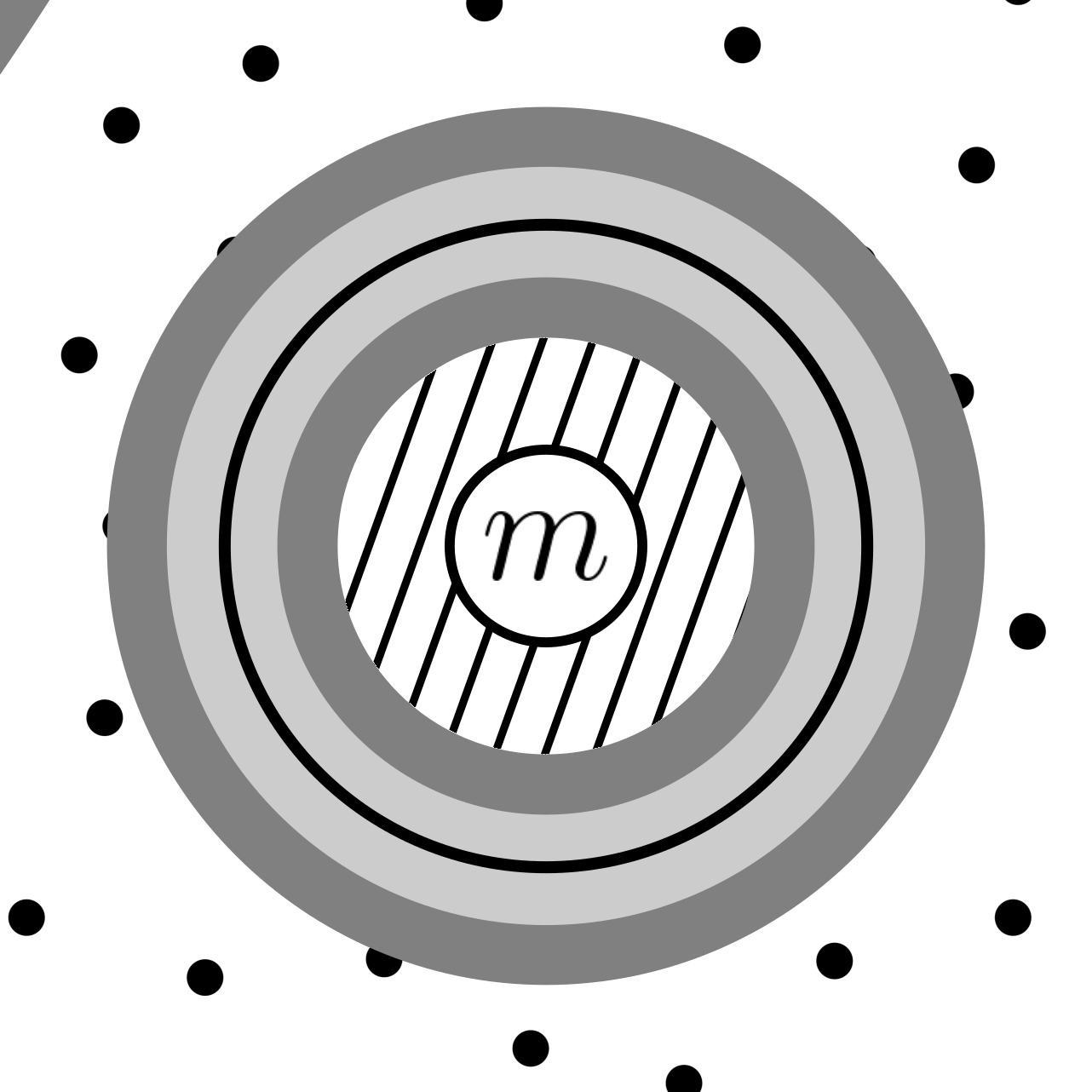}
        \caption{An illustration of the construction of $y$ near a tile center. The marker pattern $m_i$ corresponding to the shape $S_0^{(i)}$ is placed directly over the marker substrate, the pattern $y^{(m)}(M^3)$.}
        \label{fig:construction_marker}
    \end{figure}
    
    The point $y$ is well-defined by identical argument as (and as a consequence of) the fact that $y_1$ is well-defined. Before proceeding, we note a property of $y$ which is critical to later arguments. Let $c\in C$ be arbitrary and suppose that $t_0(c) = S_0^{(i)}$ for some unique index $i \in [1,r]$. Recall $y_i \in Y$ is the point from which the marker pattern $m_i$ is drawn. Then \begin{equation*}\tag{Y1}
        \sigma^c(y)(M^3) = y_i(M^3).
    \end{equation*}
    This is true because $\sigma^c(y)(M) = m_i = y_i(M)$ by construction, together with the fact that \[
        \sigma^c(y)(M^3\setminus M) = \sigma^c(y_1)(M^3\setminus M) = y^{(m)}(M^3\setminus M) = y_i(M^3\setminus M)
    \] where above we have used the construction  of $y$, condition (\ref{eq:C2}), and Theorem~\ref{thm:the-marker-patterns}.

    Next we argue that $y$ belongs to $Y$. Recall that $Y$ is an SFT witnessed by $K$. Let $g\in G$ be arbitrary and consider the translate $Kg$. If $Kg$ intersects $Mc$ for some $c\in C$, then $Kg\subset M^3c$ by the fact that $K = K^{-1}$ and $K\subset M$. Moreover, the $c$ must be unique. Then, if $t_0(c) = S_0^{(i)}$ for some unique index $i\in [1,r]$, property (Y1) implies that $\sigma^g(y)(K) = \sigma^{gc^{-1}}(y_i)(K)$. As $y_i \in Y$, we therefore see that $\sigma^g(y)(K) \in \P(K,Y)$. In the opposite case, $Kg \subset G\setminus MC$, in which case $\sigma^g(y)(K) = \sigma^g(y_1)(K)$ by construction. Because $y_1 \in Y_1 \subset Y$, we again see that $\sigma^g(y)(K) \in \P(K,Y)$. Therefore, every pattern of shape $K$ appearing in $y$ is allowed in $Y$, thus $y\in Y$.
    
    
    
    Next we argue that the map $\psi$ is continuous and shift-commuting. This follows by identical argument as (and as a consequence of) the fact that the map $x \mapsto y_1$ is continuous and shift-commuting. 
    
    Before proving that the map $\psi$ is injective, we note one more property in advance. Let $x\in X$ be fixed, let $t_0 = \T(x)\in T_0$, let $C = C(t_0) \subset G$ and let $y = \psi(x) \in Y$. We claim that for each $g\in G$ and each index $i\in [1,r]$, we have \begin{equation*}\tag{Y2}
        \sigma^g(y)(M) = m_i \text{ if and only if } t_0(g) = S_0^{(i)} \in \S_0.
    \end{equation*} The reverse implication is obvious by construction. For the forward implication, let $g\in G$ be arbitrary and suppose $\sigma^g(y)(M) = m_i$ for some index $i\in [1,r]$. If $Mg$ is disjoint from $MC$, then $\sigma^g(y)(M) = \sigma^g(y_1)(M)$, in which case $\sigma^g(y)(M) = m_i$ contradicts the fact that $m_i$ is forbidden in $Y_1$. Therefore $Mg$ must intersect $MC$, in which case $Mg \subset M^3c$ for some $c\in C$ by the fact that $M = M^{-1}$. Moreover, the $c$ must be unique. Suppose $t_0(c) = S_0^{(j)}$ for some index $j\in [1,r]$. By (Y1) we have that $\sigma^c(y)(M^3) = y_j(M^3)$. Moreover, the fact that $\sigma^g(y)(M) = m_i$ implies that $m_i = \sigma^{gc^{-1}}(y_j)(M)$. By Theorem~\ref{thm:the-marker-patterns}, this only happens in the case where $j = i$ and $g = c$. The claim follows.
    


    Finally, we argue that $\psi$ is injective. Let $x_1$, $x_2\in X$ be arbitrary and suppose that $\psi(x_1) = \psi(x_2)$. Property (Y2) implies that $\T(x_1) = \T(x_2) = t_0 \in T_0$. Then, let $t_1$, $t_1^*$, and $t_2$ be derived from $t_0$ as before and let $C = C(t_0) \subset G$. The condition (\ref{eq:C2}) and the fact that $\Psi$ is an injective map implies that $\sigma^c(x_1)(t_2(c)) = \sigma^c(x_2)(t_2(c))$ for every $c\in C$. Because $t_2$ is an exact tiling of $G$, it follows that $x_1 = x_2$.
\end{proof}

\section{Discussion}

In Theorem~\ref{thm:main-result}, can the assumption that $Y$ contains a factor of $X$ be dropped? That is, under what conditions on $X$ and $Y$ does there necessarily exist a homomorphism $\phi : X \to Y$? This is true if for example $Y$ contains a fixed point, because every subshift factors onto a fixed point. A homomorphism was constructed by Lightwood \cite[Theorem~2.8]{lightwood-II} for $G = \mathbb{Z}^2$ in the case that $X$ is strongly aperiodic and $Y$ is an SFT which satisfies the ``square-filling mixing" condition, therefore providing an extension of Krieger's embedding theorem to $\mathbb{Z}^2$. But, the existence of such homomorphisms in general remains open, even for $G = \mathbb{Z}^d$ where $d\geq 3$.

Can the mixing condition on $Y$ be weakened? 
One might wish to replace the ``uniform" mixing condition of strong irreducibility with a non-uniform mixing condition, such as topological mixing. However, such conditions are too weak for the construction given here, as we rely on the fact that we can cut and paste patterns on tiles in $Y$ that are ``packed together" relatively tightly, with the tiling known to cover a fraction of the group which can be chosen arbitrarily close to $1$. Furthermore, it is known that topological mixing is too weak for a general embedding theorem (i.e., for a general $X$) even for $G = \mathbb{Z}^2$; indeed, Quas and \c{S}ahin \cite[Theorem~1.1]{quas-sahin} have constructed an example of a topologically mixing SFT $Y$ and a nonnegative constant $h_0 < h(Y)$ such that no subshift $X$ satisfying $h_0 \leq h(X) \leq h(Y)$ together with a stronger mixing condition (``square mixing"/``uniform filling" property) can be embedded in $Y$.

Can the assumption that $G$ has the comparison property be dropped? If there are no amenable groups without the comparison property, then this point is moot. We invoke the comparison property in two places in this proof, in each case to construct a desirable system of quasi-tilings. In the first case, as part of the construction of the subshift $Y_0$, we construct a strongly irreducible system of exact tilings of $G$ by way of Theorem~\ref{thm:quasi-to-exact-tiling-factor} (due to Downarowicz and Zhang \cite{downarowicz-zhang-symb-ext, downarowicz-zhang-comparison}) in combination with a construction of Frisch and Tamuz \cite{frisch-tamuz}. In the second case, we adapt the proof of Theorem~\ref{thm:quasi-to-exact-tiling-factor} to construct a system $T_2$ of exact tilings as a factor of a sufficiently nice system of disjoint quasi-tilings $T_1$. As mentioned before, the most significant aspect of $T_2$ is that its tilings completely cover $G$, and disjointness here could in principle be traded for near-disjointness.

If one refuses the comparison property and utilizes instead either $T_0$ or $T_1$ on the side of $X$ to construct the map $\psi : X \to Y$, then $\psi$ has the property that for every $x_1$, $x_2\in X$, if $\psi(x_1) = \psi(x_2)$ then $\T(x_1) = \T(x_2) = t_0\in T_0$ and $x_1(g) = x_2(g)$ for every $g\in \bigcup_c t_0(c)c$. This map is therefore not necessarily injective (unless $t_0$ covers $G$), but one does have that the difference in entropy between $X$ and $\psi(X) \subset Y$ can be made arbitrarily small (based on the covering density of $t_0$). Therefore in this case, if $Y$ contains at least one factor of $X$ then necessarily $Y$ contains many ``nontrivial" factors of $X$, in particular with entropy arbitrarily close to $X$.


One aspect of the quasi-tilings given by Theorem~\ref{thm:dynamical-quasitilings} that we have not exploited is that they are actually \textit{maximal}, in the sense that no one additional tile of any shape could be inserted anywhere without breaking the $\varepsilon$-disjointness property (this is seen if one closely inspects the proof of \cite[Lemma~3.4]{downarowicz-huczek}). This condition is similar to that of the $\rho$-covering condition (Definition~\ref{def:rho-covering}); while the $\rho$-covering is only in general witnessed at some ``scale" $F$ (depending on the quasi-tiling, and possibly much larger than the tiles themselves), the maximality implies that the covering is somehow witnessed on the scale of the tiles. This leverage could be useful for some applications. 


\section*{Acknowledgements}

This work is part of the author's Ph.D. dissertation, conducted under the guidance of Kevin McGoff. The author would also like to thank Mike Boyle for many helpful comments. This work was supported in part by the National Science Foundation grant DMS-1847144.


\printbibliography

\end{document}